\documentclass[10pt]{amsart}
\usepackage{amsmath}
\usepackage{amscd,amsthm,amssymb,amsfonts}
\usepackage{mathrsfs}
\usepackage{dsfont}
\usepackage{stmaryrd}
\usepackage{euscript}
\usepackage{expdlist}
\usepackage{enumerate}

\input xy
\xyoption{all}
\usepackage[OT2,T1]{fontenc}

\theoremstyle{plain}
\newtheorem{thm}{Theorem}[section]

\newtheorem*{thm*}{Theorem}

\newtheorem{lm}[thm]{Lemma}
\newtheorem{cor}[thm]{Corollary}
\newtheorem*{cor*}{Corollary}
\newtheorem{prop}[thm]{Proposition}
\newtheorem*{conj*}{Conjecture}

\newtheorem{I_thm}{Theorem} 
\newtheorem{I_cor}{Corollary} 

\theoremstyle{remark}
\newtheorem*{remark}{Remark}
\newtheorem{I_remark}{Remark} 
\newtheorem*{thank}{Acknowledgments}

\theoremstyle{definition}
\newtheorem*{defn*}{Definition}

\newtheorem{defn}[thm]{Definition}

\newtheorem{example}[thm]{Example}

\newcommand{\nc}{\newcommand}

\newcommand{\beq}{\begin{equation}}
\newcommand{\eeq}{\end{equation}}
\newcommand{\bpmx}{\begin{pmatrix}}
\newcommand{\epmx}{\end{pmatrix}}
\newcommand{\bbmx}{\begin{bmatrix}}
\newcommand{\ebmx}{\end{bmatrix}}
\newcommand{\wh}{\widehat}
\newcommand{\wtd}{\widetilde}

\newcommand{\beqcd}[1]{\begin{equation*}\label{#1}\tag{#1}}
\newcommand{\eeqcd}{\end{equation*}}

\numberwithin{equation}{section}
\newenvironment{mylist}{
  \begin{enumerate}{}{%
      \setlength{\itemsep}{5pt} \setlength{\parsep}{0in}
      \setlength{\parskip}{0in} \setlength{\topsep}{0in}
      \setlength{\partopsep}{0in}
      \setlength{\leftmargin}{0.17in}}}{\end{enumerate}}

\def\parref#1{\ref{#1}}
\def\thmref#1{Theorem~\parref{#1}}
\def\propref#1{Prop.~\parref{#1}}
\def\corref#1{Cor.~\parref{#1}}     \def\remref#1{Remark~\parref{#1}}
\def\secref#1{\S\parref{#1}}
\def\chref#1{Chapter~\parref{#1}}
\def\lmref#1{Lemma~\parref{#1}}
\def\subsecref#1{\S\parref{#1}}
\def\defref#1{Def.~\parref{#1}}
\def\assref#1{Assumption~\parref{#1}}
\def\hypref#1{Hypothesis~\parref{#1}}
\def\subsubsecref#1{\S\parref{#1}}
\def\egref#1{Example~\parref{#1}}

\def\conjref#1{Conjecture~\parref{#1}}
\def\makeop#1{\expandafter\def\csname#1\endcsname
  {\mathop{\rm #1}\nolimits}\ignorespaces}

\makeop{Hom}   \makeop{End}   \makeop{Aut}   
\makeop{Pic} \makeop{Gal}       \makeop{Div} \makeop{Lie}
\makeop{PGL}   \makeop{Corr} \makeop{PSL} \makeop{sgn} \makeop{Spf}
 \makeop{Tr} \makeop{Nr} \makeop{Fr} \makeop{disc}
\makeop{Proj} \makeop{supp} \makeop{ker}   \makeop{Im} \makeop{dom}
\makeop{coker} \makeop{Stab} \makeop{SO} \makeop{SL} \makeop{SL}
\makeop{Cl}    \makeop{cond} \makeop{Br} \makeop{inv} \makeop{rank}
\makeop{id}    \makeop{Fil} \makeop{Frac}  \makeop{GL} \makeop{SU}
\makeop{Trd}   \makeop{Sp} \makeop{Tr}    \makeop{Trd} \makeop{Res}
\makeop{ind} \makeop{depth} \makeop{Tr} \makeop{st} \makeop{Ad}
\makeop{Int} \makeop{tr}    \makeop{Sym} \makeop{can} \makeop{SO}
\makeop{torsion} \makeop{GSp} \makeop{Tor}\makeop{Ker} \makeop{rec}
\makeop{Ind} \makeop{Coker}
 \makeop{vol} \makeop{Ext} \makeop{gr} \makeop{ad}
 \makeop{Gr}\makeop{corank} \makeop{Ann}
\makeop{Hol} 
\makeop{Fitt} \makeop{Mp} \makeop{CAP}

\def\Sel{Sel}
\def\ord{{ord}}

\def\Ord{{\mathrm{ord}}}

\def\Spec{\mathrm{Spec}\,}

\DeclareMathOperator{\length}{length}
\DeclareMathAlphabet{\mathpzc}{OT1}{pzc}{m}{it}
\DeclareSymbolFont{cyrletters}{OT2}{wncyr}{m}{n}
\DeclareMathSymbol{\SHA}{\mathalpha}{cyrletters}{"58}

\def\makebb#1{\expandafter\def
  \csname bb#1\endcsname{{\mathbb{#1}}}\ignorespaces}
\def\makebf#1{\expandafter\def\csname bf#1\endcsname{{\bf
      #1}}\ignorespaces}
\def\makegr#1{\expandafter\def
  \csname gr#1\endcsname{{\mathfrak{#1}}}\ignorespaces}
\def\makescr#1{\expandafter\def
  \csname scr#1\endcsname{{\EuScript{#1}}}\ignorespaces}
\def\makecal#1{\expandafter\def\csname cal#1\endcsname{{\mathcal
      #1}}\ignorespaces}

\def\doLetters#1{#1A #1B #1C #1D #1E #1F #1G #1H #1I #1J #1K #1L #1M
                 #1N #1O #1P #1Q #1R #1S #1T #1U #1V #1W #1X #1Y #1Z}
\def\doletters#1{#1a #1b #1c #1d #1e #1f #1g #1h #1i #1j #1k #1l #1m
                 #1n #1o #1p #1q #1r #1s #1t #1u #1v #1w #1x #1y #1z}
\doLetters\makebb   \doLetters\makecal  \doLetters\makebf
\doLetters\makescr
\doletters\makebf   \doLetters\makegr   \doletters\makegr

    \def\setminus{\smallsetminus}

\normalsize

\makeop{Ram} \makeop{Rep} \makeop{mass}

\makeop{Bl}

\def\norm#1{\lVert#1\rVert}
\def\Fpbar{\bar{\mathbb F}_p}
\def\Fp{{\mathbb F}_p}

\def\Qp{\Q_p}
\def\Qbar{\bar\Q}

\def\Zp{\Z_p}

\def\rmT{{\mathrm T}}
\def\rmN{{\mathrm N}}

\def\cA{{\mathcal A}}  

\def\cD{\mathcal D}
\def\cE{{\mathcal E}}
\def\cG{{\mathcal G}}

\def\cH{{\mathcal H}}
\def\cI{\mathcal I}
\def\cJ{\mathcal J}
\def\cK{{\mathcal K}}  
\def\cM{\mathcal M}
\def\cR{{\mathcal R}}
\def\cO{\mathcal O}
\def\cS{{\mathcal S}}
\def\cf{{\mathcal f}}

\def\cX{\mathcal X}
\def\cV{{\mathcal V}}

\def\cJ{\mathcal J}

\def\cU{\mathcal U}

\def\bff{\mathbf f}

\def\bftheta{\boldsymbol{\theta}}

\def\sF{\mathscr F}
\def\sL{\mathscr L}

\def\sT{\mathscr T}

\newcommand{\Z}{\mathbf Z}
\newcommand{\Q}{\mathbf Q}
\newcommand{\R}{\mathbf R}
\newcommand{\C}{\mathbf C}
\newcommand{\A}{\mathbf A}    
\newcommand{\F}{\mathbb F}

\def\frakp{{\mathfrak p}}

\def\frakm{\mathfrak m}

\def\frakl{\mathfrak l}

\def\frakN{\mathfrak N}

\def\bfone{{\mathbf 1}}

\def\Zhat{\hat{\Z}}
\def\wbar{\bar{w}}

\def\wbar{\bar{w}}

\def\et{{\acute{e}t}}

\def\etale{{\'{e}tale }}

\def\pont{Pontryagin } 
\def\padic{\text{$p$-adic }}

\def\Neron{N\'{e}ron }
\def\Frob{\mathrm{Frob}}

\newcommand{\<}{\langle}   
\renewcommand{\>}{\rangle} 

\def\isoto{\stackrel{\sim}{\to}}

\def\surjto{\twoheadrightarrow}

\def\ot{\otimes}

\def\hookto{\hookrightarrow}
\def\longto{\longrightarrow}
\def\ol{\overline}  \nc{\opp}{\mathrm{opp}} \nc{\ul}{\underline}

\newcommand{\pair}[2]{\< #1, #2\>}

\newcommand{\pairing}{\pair{\,}{\,}}

\def\XYmatrix{\xymatrix@M=8pt} 
\def\ncmd{\newcommand}
\ncmd{\xysubset}[1][r]{\ar@<-2.5pt>@{^(-}[#1]\ar@<2.5pt>@{_(-}[#1]}
\ncmd{\XYmatrixc}[1]{\vcenter{\XYmatrix{#1}}}
\ncmd{\xyto}[1][r]{\ar@{->}[#1]}
\ncmd{\xyinj}[1][r]{\ar@{^(->}[#1]}
\ncmd{\xysurj}[1][r]{\ar@{->>}[#1]}
\ncmd{\xyline}[1][r]{\ar@{-}[#1]}
\ncmd{\xydotsto}[1][r]{\ar@{.>}[#1]}
\ncmd{\xydots}[1][r]{\ar@{.}[#1]}
\ncmd{\xyleadsto}[1][r]{\ar@{~>}[#1]}
\ncmd{\xyeq}[1][r]{\ar@{=}[#1]} \ncmd{\xyequal}[1][r]{\ar@{=}[#1]}
\ncmd{\xyequals}[1][r]{\ar@{=}[#1]}
\ncmd{\xymapsto}[1][r]{l\ar@{|->}[#1]}\ncmd{\xyimplies}[1][r]{\ar@{=>}[#1]}
\ncmd{\xyiso}{\ar[r]_-{\sim}}
\def\injxy{\ar@{^(->}}

\newcommand{\pMX}[4]{\begin{pmatrix}
{#1}& {#2}\\
{#3}&{#4}\end{pmatrix} }

 \newcommand{\pDII}[2]{\begin{pmatrix}{#1}&0
 \\0&{#2}\end{pmatrix}}

\newcommand{\seesaw}[4]{{#1}\ar@{-}[rd]\ar@{-}[d]&{#2}\ar@{-}[d]\\
{#3}\ar@{-}[ru]&{#4}}

\def\ie{i.e. }

\def\cf{\mbox{{\it cf.} }}
\def\loccit{\mbox{{\it loc.cit.} }}

\def\mapR{\smash{\mathop{\longrightarrow}}}

\newcommand{\exact}[3]{
0\mapR{#1}\mapR{#2}\mapR{#3}\mapR 0 }

\def\Ch{\mathrm{char}}

\def\uf{\varpi} 

\def\Sg{{\varSigma}}  

\def\ndivides{\nmid}

\def\x{{\times}}

\def\e{\varepsilon} 
\def\al{\alpha}

\def\Lam{\Lambda}
\def\kap{\kappa}

\def\dirlim{\varinjlim}
\def\prolim{\varprojlim}
\def\iso{\simeq}
\def\con{\equiv}
\def\bksl{\backslash}
\newcommand\stt[1]{\left\{#1\right\}}
\def\ep{\epsilon}

\def\pd{\partial}
\def\lam{\lambda}

\def\sg{\sigma}
\def\vp{\varphi}
\def\disjoint{\bigsqcup}

\def\setp{{(p)}}

\newcommand{\powerseries}[1]{\llbracket{#1}\rrbracket}

\renewcommand\pmod[1]{\,(\mbox{mod }{#1})}

\newcommand\Dmd[1]{\left<{#1}\right>} 

\def\Cp{\C_p}

\usepackage{hyperref}
\title[On the anticyclotomic Iwasawa main conjecture for modular forms]{On the anticyclotomic Iwasawa main conjecture for modular forms} 
\thanks{The first author is partly supported by
Grant-in-Aid for Young Scientists (B) No. 23740015  and Hakubi project
of Kyoto University. The second author is partially supported by National Science Council grant 101-2115-M-002-010-MY2}
\author[M. Chida]{Masataka Chida}
\address{ Department of Mathematics, Kyoto University,
Kitashirakawa Oiwake-cho, Sakyo-ku, Kyoto 606-8502.
}
\address{
The Hakubi Center for Advanced Research, Kyoto University,
Yoshida-Ushinomiya-cho, Sakyo-ku, Kyoto, 606-8302, Japan
}
\email{chida@math.kyoto-u.ac.jp}
\author[M.-L. Hsieh]{Ming-Lun Hsieh}
\address{ Department of Mathematics~\\National Taiwan University ~ \\
No. 1, Sec. 4, Roosevelt Road, Taipei 10617, Taiwan~
}
\email{mlhsieh@math.ntu.edu.tw}
\date{\today}
\begin{document}
\begin{abstract}We generalize the work of Bertolini and Darmon on the anticyclotomic main conjecture for elliptic curves to modular forms of higher weight.
\end{abstract}
\maketitle
\tableofcontents

\def\cmpt{\varsigma}
\def\ShCl{M^{[\ell]}_n}
\def\Jacl{J^{[\ell]}_n}
\def\Bl{B(\ell)}
\def\cK{K}
\def\cmJ{J}
\def\pme{q}
\def\Un{\cU}
\def\setl{{(\ell)}}
\def\wt{k}
\def\padicMF{\cM}
\def\Bhat{\wh B}
\def\frakml{\frakm^{[\ell]}}
\def\fn{g}
\def\Ill{I_{\fn}}
\def\cIll{\cI_{\fn}^{[\ell]}}
\def\cIfn{\cI_{\fn}}
\def\xx{\,\times\,}
\def\sing{{\mathrm{sing}}}
\def\ord{{\mathrm{ord}}}
\def\an{\mathrm{an}}
\def\rig{\mathrm{rig}}
\def\CM{\mathrm{CM}}
\def\red{\mathrm{red}}
\def\res{\mathrm{res}}
\def\Diff{\delta}
\def\CMP{\bftheta}
\def\Qq{\Q_\pme}
\def\Zhat{\wh\Z}
\def\XB{X_B}
\def\Sg{\Sigma}
\def\Setl{{[\ell]}}
\def\cOn{\cO_n}
\def\Tfn{T_{f,n}}
\def\Afn{A_{f,n}}
\def\Sel{\mathrm{Sel}}
\def\SB{\cS^B}
\def\SBk{\cS^B_k}
\def\divides{\,|\,}
\def\CR{$(\mathrm{CR}^+)$}
\def\cores{\mathrm{cor}}
\def\thmref#1{Theorem~\parref{#1}}
\def\propref#1{Proposition~\parref{#1}}
\def\corref#1{Corollary~\parref{#1}}     \def\remref#1{Remark~\parref{#1}}
\def\secref#1{\S\parref{#1}}
\def\chref#1{Chapter~\parref{#1}}
\def\lmref#1{Lemma~\parref{#1}}
\def\subsecref#1{\S\parref{#1}}
\def\defref#1{Definition~\parref{#1}}
\def\assref#1{Assumption~\parref{#1}}
\def\hypref#1{Hypothesis~\parref{#1}}
\def\subsubsecref#1{\S\parref{#1}}
\def\egref#1{Example~\parref{#1}}
\def\conjref#1{Conjecture~\parref{#1}}
\def\th{Theorem\,}
\def\proposition{Proposition\,}
\def\lemma{Lemma\,}
\def\conjecture{Conjecture\,}
\def\corollary{Corollary\,}
\def\example{Example\,}
\newtheorem*{hypCR}{Hypothesis \CR}
\section*{Introduction}
This is the continuation of our previous work \cite{Hsieh_Chida} on the analytic side of Iwasawa theory for modular forms over the anticyclotomic $\Zp$-extension of imaginary quadratic fields, \ie the construction of \padic $L$-functions and explicit interpolation formulas. The purpose of this article is to prove a one-sided divisibility relation towards the main conjecture in Iwasawa theory for modular forms over anticyclotomic $\Zp$-extensions by generalizing the proof of Bertolini and Darmon \cite{Bertolini_Darmon:IMC_anti} for elliptic curves. To state our result precisely, we introduce some notation. Let $f\in S_k(\Gamma_0(N))$ be an elliptic new form of level $N$ with $q$-expansion at the infinity cusp
\[f(q)=\sum_{n>0}\bfa_n(f)q^n.\]
Let $\cK$ be an imaginary quadratic field with absolute discriminant $D_\cK$. Decompose $N=N^+N^-$, where $N^+$ is only divisible by primes split in $\cK$ and $N^-$ is only divisible by primes inert or ramified in $\cK$. In this article, we assume that
\beqcd{ST}\text{ $N^-$ is the square-free product of an odd number of inert primes}.\eeqcd
Let $p$ be a distinguished rational prime such that
\[p\ndivides ND_\cK.\]
Fix an embedding $\iota_p:\Qbar\to\Cp$. Let $E= \Qp(f)$ be the Hecke field of $f$ in $\Cp$, \ie the finite extension of $\Qp$ generated by $\stt{\bfa_n(f)}_{n}$.
Let $\cO$ be the ring of integers of $E$ and $\mathbb{F}$ be the residue field. Henceforth, we assume that
\beqcd{ord}\text{$f$ is $p$-ordinary,}\eeqcd\ie the $p$-th Fourier coefficient $\bfa_p(f)$ is a unit in $\cO$. Let $\rho_{f}:G_\Q=\Gal(\Qbar/\Q)\to \GL_2(E)$ be the \padic Galois representation attached to $f$. We have $\det\rho_f=\e^{k-1}$, where $\e:G_\Q\to \Zp^\x$ is the \padic cyclotomic character. We consider the self-dual Galois representation \beq\label{Twist}\rho_f^*:=\rho_{f}\ot\e^\frac{2-k}{2}:G_\Q\to\GL_2(E).\eeq
Let $V_f=E^{\oplus 2}$ be the representation space of $\rho_f^*$. By the ordinary assumption for $f$, there exists a unique rank one $G_{\Qp}$-invariant subspace $F_p^+V_f\subset V_f$ on which the inertia group of $G_{\Qp}$ acts via $\e^\frac{k}{2}$.
We shall fix a $G_\Q$-stable lattice $T_f \subset V_f$ once and for all. Let $A_f:=V_f/T_f$ and let $F_p^+A_f$ be the image of $F_p^+V_f$ in $A_f$.
Let $\cK_\infty$ be the anticyclotomic $\Zp$-extension of $\cK$. Let $\Gamma=\Gal(\cK_\infty/\cK)$ and let $\Lam=\cO\powerseries{\Gamma}$ be the one-variable Iwasawa algebra over $\cO$. In this paper, we are interested in the $\Lam$-adic \emph{minimal} Selmer group $\Sel(K_\infty,A_f)$ for $\rho_f^*$. Recall that for each algebraic extension $L$ over $K$, the minimal Selmer group $\Sel(L,A_f)$ is defined by
\[\Sel(L,A_f):=\ker\stt{H^1(L,A_f)\to \prod_{v\ndivides p}H^1(L_{v},A_f)\x \prod_{v|p}H^1(L_{v},A_f/F_p^+A_f)},\]
where $v$ runs over places of $L$ and $L_{v}$ is the completion of $L$ with respect to $v$. It is well-known that the \pont dual $\Sel(K_\infty,A_f)^\vee$ of $\Sel(K_\infty,A_f)$ is a finitely generated $\Lam$-module.

On the other hand, in \cite[\th A]{Hsieh_Chida} we construct a theta element $\theta_\infty\in\Lam$ obtained by the evaluation of a $p$-ordinary definite quaternionic modular form at Gross points, and define a complex number $\Omega_{f,N^-}\in\C^\x$ attached to $(f,N^-)$ such that for every finite order character $\chi:\Gamma\to\mu_{p^\infty}$ of conductor $p^n$, the anticyclotomic \padic $L$-function $L_p(K_\infty,f):=\theta_\infty^2$ satisfies the following interpolation formula:
\beq\label{E:interpolation}\begin{aligned}\chi(L_p(K_\infty,f))=&\Gamma(\frac{k}{2})^2\cdot \frac{L(f/\cK,\chi,\frac{k}{2})}{\Omega_{f,N^-}} \cdot e_p(f,\chi)^{2}\cdot p^n \al_p(f)^{-2n}(p^{n}D_\cK)^{k-2} \cdot u_\cK^2\sqrt{D_\cK},\end{aligned}\eeq
where $\al_p(f)$ is the \padic unit root of the Hecke polynomial $X^2-\bfa_p(f)X+p^{k-1}$, $u_\cK=\#(\cO_\cK^\x)/2$ and $e_p(f,\chi)$ is the \padic multiplier defined by
 \[e_p(f,\chi)=\begin{cases}1&\text{if $n>0$,}\\
(1-\chi(\frakp)p^\frac{k-2}{2}\al_p(f)^{-1})(1-\chi(\ol{\frakp})p^\frac{k-2}{2}\al_p(f)^{-1})&\text{if $n=0$ and $p=\frakp\ol{\frakp}$ is split},\\
1-p^{k-2}\al_p(f)^{-2}&\text{if $n=0$ and $p=\frakp$ is inert}.
\end{cases}\]
In general, this complex number $\Omega_{f,N^-}$ belongs to $\Omega_f\cdot\cO$, where $\Omega_{f}$ is Hida's canonical period. Recall that \[\Omega_{f}=\frac{4^{k-1}\pi^k\norm{f}_{\Gamma_0(N)}}{\eta_f(N)},\]
 where $\norm{f}_{\Gamma_0(N)}$ is the Petersson norm of $f$ and $\eta_f(N)$ is the congruence number of $f$ among forms in $S_k(\Gamma_0(N))$.

The main result of this paper is to prove under certain hypotheses a one-sided divisibility result towards the anticyclotomic main conjecture asserting an equality between the characteristic power series $\Ch_\Lam\Sel(\cK_\infty,A_f)^{\vee}$ and the \padic $L$-function $L_p(K_\infty,f)$. To describe our hypotheses explicitly, we need to introduce some notation. Let $\bar\rho_{f}$ be the residual Galois representation of $\rho_{f}$. Throughout, we assume that the prime $p$ satisfies the following hypothesis:
\begin{hypCR}\label{H:CR}
\begin{mylist}
\item $p>k+1$ and  $\#(\Fp^\x)^{k-1}>5$,
\item The restriction of $\bar\rho_{f}$ to the absolute Galois group of $\Q(\sqrt{(-1)^\frac{p-1}{2}p})$ is absolutely irreducible,
\item $\bar\rho_f$ is ramified at $\ell$ if either of the following holds: \[{\rm(i)}\,\ell\mid N^-\text{ and }\ell^2\con 1\pmod{p},\quad {\rm(ii)}\,\ell\divides N^+.\]
\end{mylist}
\end{hypCR}
We will further assume 
\beqcd{PO}\text{ $\bfa_p(f)^2\not \con 1\pmod{p}$ if $k=2$.}\eeqcd
\begin{I_remark}\begin{mylist}\item Under the hypothesis \CR, it is proved in \cite[\proposition 6.1]{Hsieh_Chida} that \[\Omega_{f,N^-}=u\cdot \Omega_f\text{ for some }u\in\cO^\x\] if we further assume $\bar\rho_f$ is ramified at all primes dividing $N^-$.
\item \CR (2) implies that the definition of the Selmer group $\Sel(K_\infty,A_f)$ does not depend on the choice of the lattice $T_f$.
\item  Note that \eqref{PO} is indeed equivalent to saying that $\alpha_p(f)^2\not \con 1\pmod{p}$.
Moreover, it implies that $e_p(f,\bfone)\not\con 0\pmod{p}$, where $\bfone$ is the trivial character. When $f$ is attached to an elliptic curve over $\Q$, the same hypothesis is also used in \cite[Assumption 2.15 and Proposition 2.16] {BD:Duke:DerivedMT}.
\item Since $N^->1$ is square-free, $f$ can not be a CM form, and hence the hypothesis \CR\, holds for all but finitely many primes $p$ (but it is not known if there are infinitely many ordinary primes for a given modular form $f$). 
\item The assumption that $\bar\rho_f$ is ramified at $\ell\divides N^+$ can be weakened by the work \cite{KPW}.
\end{mylist}
\end{I_remark}

\begin{I_thm}\label{T:A}With the hypotheses \CR\,and \eqref{PO} for the prime $p$, we have
\[\Ch_\Lam\Sel(\cK_\infty,A_f)^{\vee}\supset (L_p(K_\infty,f)).\]
\end{I_thm}
\begin{I_remark}\begin{enumerate}\item For the case $k=2$, this theorem was proved by Bertolini and Darmon \cite{Bertolini_Darmon:IMC_anti} with the hypotheses for $f$ being $p$-isolated and the maximality of the image of the residual Galois representation $\bar\rho_f$. The former assumption was removed by Pollack and Weston \cite{Pollack_Weston:AMU}. We remove the assumption on the image of the residual Galois representation by looking carefully into the Euler system arguments in \cite{Bertolini_Darmon:IMC_anti}.
\item It is expected that the other divisibility follows from the work of Skinner-Urban \cite{SU:IMC_GL(2)} on the three-variable main conjecture for $f$ together with the generalization of Vastal's result on the vanishing of $\mu$-invariant of the \padic $L$-function $L_p(K_\infty,f)$ (\cite{Vatsal:nonvanishing} and \cite{Hsieh_Chida}).
\end{enumerate}
\end{I_remark}
\def\Tam{\mathrm{Tam}}
We obtain the following immediate consequence of \thmref{T:A} and \cite[Theorem C]{Hsieh_Chida}.
\begin{I_cor}With the hypotheses in \thmref{T:A}, the $\Lam$-module $\Sel(K_\infty,A_f)$ is cotorsion and its $\mu$-invariant vanishes.
\end{I_cor}
Combined with control theorems of Selmer groups and the interpolation formula of $L_p(K_\infty,f)$, the above theorem yields the following consequence:
\begin{I_cor}\label{C:A}With the hypotheses in \thmref{T:A}, if the central $L$-value $L(f/K,\frac{k}{2})$ is non-zero, then the minimal Selmer group $\Sel(K,A_f)$ is finite and
\begin{align*}\Fitt_\cO(\Sel(K,A_f))\cdot \prod_{\ell|N^+}\Tam_\ell(f)\supset&\left(\frac{L(f/\cK,\frac{k}{2})}{\Omega_{f,N^-}}\right)\cO,\end{align*}
where $\Tam_\ell(f)=\Fitt_\cO H^1(K_\ell^{ur},T_f)^{G_{K_\ell}}_{\mathrm{tor}}$ are the local Tamagawa ideals at $\ell|N^+$.
\end{I_cor}
The reader might be aware of the missing local Tamagawa ideals at $\ell|N^-$ in view of the Bloch-Kato conjecture. This discrepancy is due to the complex number $\Omega_{f,N^-}$ is different from the canonical period $\Omega_f$ in general. When $k=2$ and $N$ is square-free, it is proved in \cite{Pollack_Weston:AMU} that the ratio $\Omega_{f,N^-}/\Omega_{f}$ is precisely a product of local Tamagawa ideals at $\ell|N^-$.

The proof of \thmref{T:A} relies on the existence of Euler system together with the first and second explicit reciprocity laws \`{a} la Bertolini-Darmon for the Galois module $T_f$. This is a generalization of the construction in \cite{Bertolini_Darmon:IMC_anti} for elliptic new forms of weight two. To explain the main idea of the construction of the Euler system, we introduce some notation. Let $\uf$ be a uniformizer of $\cO$ and let $\Lam_n=\Lam\ot_\cO\cO/\uf^n\cO$ for a positive integer $n$. Let $\Tfn=T_f/\uf^nT_f$ and $\Afn=\ker\stt{\uf^n:A_f\to A_f}$. Following Bertolini and Darmon, we construct by the technique of level-raising the Euler system $\cE_n$ for each $n$ arising from Heegner points in various Shimura curves with wildly ramified level at $p$. This Euler system $\cE_n$ is a collection of norm-compatible cohomology classes $\kappa_\cD(\ell)_m\in H^1(K_m,\Tfn)$ for $K_\infty/K_m/K$ indexed by $(\ell,\cD)$, where $\ell$ is an \emph{$n$-admissible prime} for $f$ (\defref{D:admissibleprimes}) and $\cD=(\Delta,g)$ is an \emph{$n$-admissible form}, a pair consisting of $\Delta=N^-\cdot S$ with $S$ a square-free product of an even number of $n$-admissible primes and a weight two $p$-ordinary eigenform $\fn$ on the definite quaternion algebra of discriminant $\Delta$ such that Hecke eigenvalues of $\fn$ are congruent to those of $f$ modulo $\uf^n$ (\defref{D:1}).
To each $n$-admissible form $\cD=(\Delta,g)$, we can associate a finitely generated compact $\Lam_n$-module $X_{\cD}$, the \pont dual of the $\Delta$-ordinary Selmer group for $\Tfn$, and a theta element $\theta_{\cD}\in\Lam_n$ obtained by the evaluation of $g$ at Gross points. The first reciprocity law gives a connection between the Euler system $\kappa_\cD(\ell)$ and the theta element $\theta_\cD$, by which one can control the Selmer group $X_\cD$ in terms of $\theta_\cD$, and the second reciprocity law is a kind of level raising argument at two primes, which provides a decreasing induction on theta elements (or rather \padic $L$-functions). The main novelty in this paper is to establish the connection between this Euler system $\cE_n$ and the \padic $L$-function $L_p(K_\infty,f)$ of $f$ modulo $\uf^n$ by the congruence between theta elements attached to weight two forms and higher weight forms. A key observation is that when $\Delta=N^-$, we can construct an $n$-admissible form $\cD_0=(N^-,g_0)$ such that $X_{\cD_0}=\Sel(K_\infty,A_f)^\vee\pmod{\uf^n}$ and $\theta_{\cD_0}^2\con L_p(K_\infty,f)\pmod{\uf^n}$ (\propref{P:R1}). We remark that we do not make use of the congruence among (definite quaternionic) modular forms of different weights but rather we exploit the congruence between the evaluations of modular forms of weight two and higher weight at Gross points. Thus, our approach does not provide perspective for the two variable main conjecture in \cite{LongoVigni:HidaFamilies} by varying $f$ in Hida families, despite the fact that Hida theory is also a key tool used to avoid the technical difficulties arising from the use of Shimura curves with wildly ramified level at $p$ in our proof.

The hypothesis \CR\,is responsible for a freeness result of the space of definite quaternionic modular forms as Hecke modules in \cite[\proposition 6.1]{Hsieh_Chida}. The application of this freeness result is twofold. On the algebraic side, it is used crucially in the level raising argument for the construction of Euler system and in the proof of second reciprocity law. On the analytic side, it implies the equality between two periods $\Omega_f$ and $\Omega_{f,N^-}$ up to a \padic unit.The assumption \eqref{PO} roughly says $f$ is not congruent to an eigenform which is Steinberg at $p$. It is needed for the application of the version of Ihara's lemma proved in \cite{Diaomond_Taylor:Inventione}. The hypothesis \CR\,is an analogue of (CR) in \cite{Pollack_Weston:AMU} for the weight two case. It might be weakened by a careful analysis of the method of the proof of \cite[\proposition 6.1]{Hsieh_Chida}. However, it seems difficult to remove \eqref{PO} unless one works out \padic Hodge theory used in \cite{Diaomond_Taylor:Inventione} for the case of semi-stable reduction.

Since we work in the higher weight situation, one might look for the Euler system arising from CM cycles on Kuga-Sato varieties over Shimura curves. Indeed, the first author in \cite{Chida;CM-cycle} adopts this construction and proves the first explicit reciprocity law for this Euler system without the ordinary assumption for $p$. However, there remain issues to be addressed in arithmetic geometry in order to prove the second explicit reciprocity law by a direct generalization of the proof of Bertolini-Darmon, 
 where the crucial ingredient is the surjectivity of the Abel-Jacobi map from the supersingular part of the Jacobian of a Shimura curve over finite fields to the unramified part of Galois cohomology $H_{\mathrm{fin}}^1(K_\ell,\Tfn)$ for $n$-admissible primes $\ell$. In this case, Bertolini and Darmon are able to reduce this surjectivity to a version of Ihara's lemma proved in \cite{Diaomond_Taylor:Inventione}, while in the higher weight case, we do not even know the surjectivity of the Abel-Jacobi map from the Chow groups to the unramified part of Galois cohomology.

This article is organized as follows. In \secref{S:Selmer}, we recall basic facts such as control theorems for various Selmer groups. In \secref{S:modularForms} and \secref{S:ShimuraCurve}, we review the theory of \padic modular forms on definite quaternion algebras and Shimura curves. In \secref{S:EulerSystem} and \secref{S:Reciprocity}, we give the construction of the Euler system and the proof of the explicit reciprocity laws. In \secref{S:Euler}, we carry out the Euler system argument and prove the main results.

\subsubsection*{Notation}
We fix once and for all an embedding
$\iota_\infty:\Qbar\hookto\C$ and an isomorphism
$\iota:\C\iso\C_\ell$ for each rational prime $\ell$, where $\C_\ell$ is the completion of the algebraic closure of $\Q_\ell$. Let
$\iota_\ell=\iota\iota_\infty:\Qbar\hookto\C_\ell$ be their composition. Let $\Ord_\ell:\C_\ell\to\Q\cup\stt{\infty}$ be the $\ell$-adic valuation on $\C_\ell$ normalized so that $\Ord_\ell(\ell)=1$.

Denote by $\Zhat$ the profinite completion of the ring $\Z$ of rational integers. For each place $q$, denote by $\Z_q$ the $q$-adic completion of $\Z$. If $M$ is an abelian group, let $\wh M=M\ot_\Z \Zhat$ and $M_\pme=M\ot_\Z\Z_\pme$. If $R$ is a commutative ring, let $M_R=M\ot_\Z R$.

If $L$ is a number field or a local field, denote by $\cO_L$ the ring of integers of $L$ and by $G_L$ the absolute Galois group. If $L$ is a local field, denote by $I_L$ the inertia group and by $L^{ur}$ the maximal unramified extension of $L$.

For a locally compact abelian group $S$, we denote by $S^\vee$ the \pont dual of $S$.

The letter $\ell$ always denotes a rational prime.

We will retain the notation in the introduction. In this article, in addition to \eqref{ST} and \eqref{ord}, we will assume the prime $p\ndivides ND_K$ satisfies
\beqcd{CR${}^+$1}\text{$p>k+1$ and $\#(\Fp^\x)^{k-1}>5$.}\eeqcd


\newcommand{\xto}[1]{\xrightarrow{#1}}
\newcommand{\incl}{\hookrightarrow}
\newcommand{\surj}{\twoheadrightarrow}
\renewcommand{\id}{\mathrm{id}}
\newcommand{\new}{\mathrm{new}}
\newcommand{\fin}{\mathrm{fin}}
\newcommand{\tilH}{\widetilde{H}}
\newcommand{\JF}{J^{(\ell_1)}(\F_{\ell_2^2})}
\newcommand{\An}{A_n}

\section{Selmer groups}\label{S:Selmer}
\subsection{Galois cohomology groups}\label{SS:GaloisCohomology}Let $L$ be an algebraic extension of $\Q$. For a discrete $G_L$-module $M$, we put
$$
H^1(L_\ell,M)=\bigoplus_{\lambda | \ell} H^1(L_{\lambda},M), \,
H^1(I_{L_\ell},M)=\bigoplus_{\lambda | \ell} H^1(I_{L_\lambda},M),
$$
where $\lambda$ runs over all primes of $L$ dividing $\ell$. Denote by $\res_\ell:H^1(L,M)\to H^1(L_\ell,M)$ the restriction map at $\ell$. Define the finite part of $H^1(L_\ell,M)$ by
$$
H^1_{\fin}(L_\ell,M)=\ker\stt{
H^1(L_\ell,M)\to H^1(I_{L_\ell},M)},
$$
and define the singular quotient of $H^1(L_{\ell},M)$ by
$$
H^1_{\sing}(L_{\ell},M)=H^1(L_{\ell},M)/H^1_{\fin}(L_{\ell},M).
$$
The natural map induced by the restriction $\partial_\ell:H^1(L,M)\to H^1_{\sing}(L_\ell,M)$ is called the \emph{residue map}. For $\kappa\in H^1(L,M)$ with $\partial_\ell(\kappa)=0$, we let \[v_{\ell}(s)\in H^1_\fin(L_\ell,M)\] denote the image of $\kappa$ under the restriction map at $\ell$.
\subsection{Selmer groups}We will retain the notation in the introduction. Recall that we work with the self-dual Galois representation $\rho^*_f:=\rho_{f}\ot\e^\frac{2-k}{2}:G_\Q\to \Aut_\cO T_f$ attached to the new form $f\in S_k^{\mathrm{new}}(\Gamma_0(N))$. Then it is known that $\rho_f^*$ satisfies
\begin{mylist}
\item $\rho_f^*$ is unramified outside $pN$.
\item The restriction of $\rho_f^*$ to $G_{\Qp}$ is of the form $\pMX{\chi_p^{-1}\e^\frac{k}{2}}{*}{0}{\chi_p\e^\frac{2-k}{2}}$, where $\chi_p$ is unramified and $\chi_p(\Frob_p)=\al_p(f)$.
\item For all $\ell$ dividing $N$ exactly, the restriction to $\rho_f^*$ to $G_{\Q_\ell}$ is of the form $\pMX{\pm\e}{*}{0}{\pm\bfone}$, where $\bfone$ denotes the trivial character.
\end{mylist}
Here the third property is a result of Carayol \cite{Carayol:GaloisHMF}.

Fix a uniformizer $\uf$ of $\cO$. Let $n$ be a positive integer. Let $\Tfn=T_f \slash \varpi^n T_f$ and $\Afn=\ker\stt{\uf^n:A_f\to A_f}$. As $G_\Q$-modules, we have $\Tfn\iso\Afn$ and $\Afn\iso\Hom_\cO(\Tfn,E/\cO(1))$.
\begin{defn}\label{D:admissibleprimes}
A rational prime $\ell$ is said to be $n$-admissible for $f$ if it satisfies the following conditions:
\begin{enumerate}
\item $\ell$ does not divide $pN$,
\item $\ell$ is inert in $K\slash \Q$,
\item $p$ does not divide $\ell^2-1$,
\item $\varpi^n$ divides $\ell^{\frac{k}{2}}+\ell^{\frac{k-2}{2}}-\epsilon_\ell \bfa_{\ell}(f)$ with $\epsilon_\ell \in \{ \pm 1 \}$.
\end{enumerate}
\end{defn}
Let $L/K$ be a finite extension. In what follows, we introduce a $G_{\Q_\ell}$-invariant submodule $F^+_\ell \Afn$ and define the ordinary part of $H^1(L_\ell,\Afn)$ for $\ell$ a prime factor of $pN^-$ or an $n$-admissible prime. For the prime $p$, we let $F_p^+V_f\subset V$ be the $E$-rank one $G_{\Qp}$-invariant subspace on which the inertia group $I_{\Qp}$ acts via $\e^{\frac{k}{2}}$. Let $F_p^+A_f$ be the image of $F_p^+V_f$ in $A_f$ and $F_p^+\Afn=F_p^+A_f\cap \Afn$. If $\ell\divides N^-$, we let $F_\ell^+V_f\subset V_f$ be the unique $E$-rank one subspace on which
$G_{\Q_\ell}$ acts by either $\e$ or $\e\tau_\ell$, where $\e$ is the $p$-adic cyclotomic character and $\tau_\ell$ is the non-trivial unramified quadratic character of $G_{\Q_\ell}$. Let $F_\ell^+A_f$ be the image of $F_\ell^+V_f$ in $A_f$ and $F_\ell^+\Afn=F_\ell^+A_f\cap \Afn$. If $\ell$ is $n$-admissible, then $V_f$ is unramified at $\ell$ and $(\Frob_\ell-\ep_\ell)(\Frob_\ell-\ep_\ell\ell)=0$ on $\Afn$ for the Frobenius $\Frob_\ell$ of $G_{\Q_\ell}$. We let $F_\ell^+\Afn\subset \Afn$ be the unique $\cO/(\uf^n)$-corank one submodule on which $\Frob_\ell$ acts via the multiplication by $\ep_\ell \ell$. In all cases, define the ordinary part of $H^1(L_\ell,\Afn)$ by
\[
H^1_{\Ord}(L_\ell,\Afn)=\ker\stt{H^1(L_\ell,\Afn)\to H^1(L_\ell,\Afn/F_\ell^+\Afn)}.
\]
The ordinary part $H^1_\Ord(L_\ell,\Tfn)$ of $H^1(L_\ell,\Tfn)$ can be defined in the same way.

Let $\Delta$ be a square-free integer such that $\Delta/N^-$ is a product of $n$-admissible primes (so $N^-\divides \Delta$) and let $S$ be a square-free integer with $(S,p\Delta N)=1$.
\begin{defn}
For $M=\Afn$ or $\Tfn$, we define the $\Delta$-ordinary Selmer group $\Sel_\Delta^S(L,M)$ attached to $(f,n,\Delta,S)$ to be the group of elements $s$ in $H^1(L,M)$ satisfying \begin{enumerate}
\item $\pd_\ell(s)=0$ for all $\ell\ndivides p\Delta S$,
\item  $\res_\ell(s)\in H^1_\Ord(L_\ell,M)$ at the primes $\ell\divides p\Delta$,
\item $\res_\ell(s)$ is arbitrary at the primes $\ell\divides S$.
\end{enumerate}
When $S=1$, we simply write $\Sel_\Delta(L,M)$ for $\Sel^1_\Delta(L,M)$.
\end{defn}

Let $m$ be a non-negative integer. Denote by $H_m$ the ring class field of conductor $p^m$. Let $G_m=\Gal(H_m/\cK)$ and let $H_\infty=\cup_{m=1}^\infty H_m$. Let $K_m=K_\infty\cap H_m$ and $\Gamma_m=\Gal(\cK_m/\cK)$. Then $K_\infty=\cup_{m=1}^\infty K_m$ and $\Gamma =\prolim_m\Gamma_m$. Let $\Lambda = \cO\powerseries{\Gamma}= \varprojlim_m \cO [\Gamma_m]$ be the one-variable Iwasawa algebra over $\cO$ and let $\frakm_\Lam$ be the maximal ideal of $\Lam$.
Let $H^1(K_\infty, \Afn)=\varinjlim_m H^1(K_m, \Afn)$ and
$\widehat{H}^1(K_\infty,\Tfn)=\varprojlim_m H^1(K_m,\Tfn)$,
where the injective limit is taken with respect to the restriction maps and
the projective limit is taken with respect to the corestriction maps. The local cohomology groups $H^1_{\bullet}(K_{\infty,\ell},\Afn)$ and $\wh H^1_{\bullet}(K_{\infty,\ell},\Tfn)$ for $\bullet\in\stt{\fin,\sing,\Ord}$ are defined in the same way. We define
\begin{align*}\Sel_\Delta(K_\infty,\Afn)=&\dirlim_m \,
\Sel_\Delta(K_m,\Afn);\\
\wh \Sel_\Delta^S(K_\infty, \Tfn)=&\prolim_m \,
\Sel_\Delta^S(K_m,\Tfn).\end{align*}
Then $\Sel_\Delta(K_\infty,\Afn)$ (resp. $\wh \Sel_\Delta^S(K_\infty, \Tfn)$) is a discrete (resp. compact) $\Lam$-module induced by the standard $\Gamma$-action. Recall that $\Sel(K_\infty,A_f)$ is the minimal Selmer group in the introduction.
\begin{prop}\label{P:3}Suppose that \CR\,holds. Then $\Sel_{N^-} (K_\infty,A_f)=\varinjlim_n\Sel_{N^-} (K_\infty,\Afn)$ is a $\Lam$-submodule of $\Sel(K_\infty,A_f)$ with finite index. If $\Sel_{N^-} (K_\infty,A_f)$ is $\Lam$-cotorsion, then $\Sel_{N^-}(K_\infty,A_f)=\Sel(K_\infty,A_{f})$. \end{prop}
\begin{proof} This is shown in the proof of \cite[Proposition\,3.6]{Pollack_Weston:AMU}. Note that $\Sel(K_\infty,f_n)$ \loccit is precisely our $\Sel_{N^-}(K_\infty,\Afn)$.
\end{proof}
\subsection{Local cohomology groups and local Tate duality}We gather some standard results on the local cohomology groups in this subsection.
\begin{lm}\label{L:2}Suppose that $\ell\not =p$. Then
\begin{enumerate}
\item
If $\ell$ is split in $K\slash \Q$, then
$$\widehat{H}^1_{\sing}(K_{\infty,\ell},\Tfn)=\stt{0}\,,\,
H^1_{\fin}(K_{\infty,\ell},\Afn)=\stt{0}.
$$
\item
If $\ell$ is non-split in $K\slash \Q$, then
$$
\widehat{H}^1_{\sing}(K_{\infty,\ell},\Tfn)
\iso H^1_{\sing}(K_{\ell},\Tfn)\otimes \Lambda
$$
and
$$
H^1_{\fin}(K_{\infty,\ell},\Afn)\iso
\Hom (H^1_{\sing}(K_{\ell},\Tfn)\otimes \Lambda, E \slash \cO).
$$
\end{enumerate}
\end{lm}
\begin{proof}
This is \cite[Lemma 2.4, Lemma 2.5]{Bertolini_Darmon:IMC_anti}. The argument there holds even if $\ell\divides N$.
\end{proof}

\begin{lm}\label{L:rank-one} If $\ell$ is $n$-admissible, then
\begin{enumerate}
\item
$H^1_{\sing}(K_{\ell},\Tfn)$ and
$H^1_{\fin}(K_{\ell},\Tfn)$
are both isomorphic to $\cO \slash \varpi^n \cO$.
\item
$\wh H^1_{\sing}(K_{\infty,\ell},\Tfn)$ and
$\wh H^1_{\fin}(K_{\infty,\ell},\Tfn)$
are free of rank one over $\Lambda \slash \varpi^n \Lambda$.
\item We have the decompositions \begin{align*}\wh H^1(K_{\infty,\ell},\Tfn)=\wh H^1_{\fin}(K_{\infty,\ell},\Tfn)\oplus \wh H^1_{\Ord}(K_{\infty,\ell},\Tfn);\\
H^1(K_{\infty,\ell},\Afn)=H^1_{\fin}(K_{\infty,\ell},\Afn)\oplus H^1_{\Ord}(K_{\infty,\ell},\Afn).
\end{align*}
\end{enumerate}
\end{lm}
\begin{proof}
Since $\Afn$ is unramified at $\ell$, we have a direct sum $\Afn=F_\ell^+\Afn\oplus F_\ell^-\Afn$ as $G_{\Q_\ell}$-modules, where $F_\ell^-\Afn=(\Frob_\ell-\ep_\ell\ell)\Afn$. An easy calculation shows that
\[H^1_{\fin}(K_{\ell},\Afn)=H^1(K_\ell^{ur}/K_\ell,\Afn)=H^1(K_\ell^{ur}/K_\ell,F_\ell^-\Afn)\iso F_\ell^-\Afn.\]
We thus have
\begin{align*}H^1(K_{\ell},\Afn)=&H^1(K_{\ell},F^-_\ell\Afn)\oplus H^1(K_{\ell},F_\ell^+\Afn)\\
=&H^1_{\fin}(K_{\ell},\Afn)\oplus H^1_\Ord(K_{\ell},\Afn).
\end{align*}
Combined with \lmref{L:2} (2), the lemma follows immediately. \qedhere
\end{proof}

Since $\Afn$ and $\Tfn$ are isomorphic to their Cartier duals, the pairing induced by the cup product on the Galois cohomology gives rise to the collection of local Tate pairings at the primes above $\ell$ over $K_m$:
$$
{ \langle \, , \,  \rangle }_{m,\ell}:H^1(K_{m,\ell}, \Tfn)\times H^1(K_{m,\ell}, \Afn)\to E \slash \cO .
$$
Taking the limit with $m$, we obtain a perfect pairing
$$
{ \langle \, , \,  \rangle }_{\ell}:
\widehat{H}^1(K_{\infty,\ell}, \Tfn)\times H^1(K_{\infty,\ell}, \Afn) \to E \slash \cO .
$$
These pairings are compatible with the action of $\Lambda$, so they induce an isomorphism of $\Lam$-modules
\[\widehat{H}^1(K_{\infty,\ell}, \Tfn) \iso H^1(K_{\infty,\ell}, \Afn)^\vee.\]
The following result is well-known.
\begin{prop}\label{P:annihilator}Suppose that $\ell\not =p$. 
$H^1_{\fin}(K_{\infty,\ell},\Afn)$ and $\widehat{H}^1_{\fin}(K_{\infty,\ell},\Tfn)$ are orthogonal complements under the pairing $\pairing_{\ell}$. In particular, $H^1_{\fin}(K_{\infty,\ell},\Afn)$ and $\widehat{H}^1_{\sing}(K_{\infty,\ell},\Tfn)$
are the Pontryagin dual of each other.
\end{prop}
\begin{proof}This is well-known. For example, see \cite[Proposition\,1.4.3]{Rubin:EulerSystem}.\end{proof}

\begin{prop}\label{P:annihilator_ord}Suppose that $\ell$ is $n$-admissible or $\ell$ divides $pN^-$. Assume further that \eqref{PO} holds and
\beqcd{CR${}^+$3}\text{$\bar\rho_f$ is ramified at $\ell$ if $\ell\divides N^-$ and $\ell^2\con 1\pmod{p}$.}\eeqcd
Then we have \begin{mylist} \item
$H^1_{\ord}(K_{\infty,\ell},\Afn)$ and $\widehat{H}^1_{\ord}(K_{\infty,\ell},\Tfn)$ are orthogonal complements under the local Tate pairing. \item If $\ell$ is $n$-admissible, then $\wh H^1_\fin(K_{\infty,\ell},\Tfn)$ and $H^1_{\Ord}(K_{\infty,\ell},\Afn)$ are the \pont dual of each other.\end{mylist}
\end{prop}
\begin{proof}The assertions for $n$-admissible primes follow from \lmref{L:rank-one} and \propref{P:annihilator}. Let $\ell\divides N^-$. We prove part (1) for $\ell$.
A simple calculation shows that
\[\#(H^1(K_\ell,\Afn))=\#(\Afn^{G_{K_\ell}})^2,\quad \#(H^1_\Ord(K_\ell,\Afn))=\#(\Afn^{G_{K_\ell}}).\]
Therefore, it suffices to show $H^1_\Ord(K_\ell,\Afn)$ and $H^1_\Ord(K_\ell,\Tfn)$ are orthogonal to each other. If $\bar\rho_f$ is unramified at $\ell$, then $\ell^2\not\con 1\pmod{p}$ by \eqref{CR${}^+$3}, and the orthogonality follows from $H^2(K_\ell,\cO/\uf^n\cO(\e^2))=\stt{0}$. If $\bar\rho_{f}$ is ramified at $\ell$, then \[H^1(K^{ur}_\ell/K_\ell,F_\ell^+\Afn)\isoto H^1_\fin(K_\ell,\Afn)\hookto H^1_\Ord(K_\ell,\Afn).\] In addition, it is to easy to see $\#(H^1_\fin(K_\ell,\Afn))=\#(\Afn^{G_{K_\ell}})$.
This shows $H^1_\fin(K_\ell,\Afn)=H^1_\Ord(K_\ell,\Afn)$, and the assertion follows from \propref{P:annihilator}.

 We consider the case $\ell=p$. Let $L/K$ be a finite extension in $K_\infty$. By the following \lmref{L:5}, we have the duality $(F_p^+\Afn)^\vee(1)\iso \Afn/F_p^+\Afn\iso \cOn(\chi_p\e^\frac{2-k}{2})$ and by a simple calculation, we find that $H^1(L_p,F_p^+\Afn)\isoto H^1_\Ord(L_p,\Afn)$ and an exact sequence
\[0\to H^1(L_p,F_p^+\Afn)\to H^1(L_p,\Afn)\to H^1(L_p,\Afn/F_p^+\Afn)\to 0.\]
The assertion for $\ell=p$ now follows from $H^2(L_p, \cOn(\e^k))=\stt{0}$ for $k>1$.
\end{proof}

\begin{lm}\label{L:5} Suppose that \eqref{PO} holds. Then we have $H^0(L_p,\Afn/F_p^+\Afn)=\stt{0}$.
\end{lm}
\begin{proof}
Let $\Tfn^-:=\Tfn/F_p^+\Tfn$. Let $v$ be a place of $L$ above $p$. Since $L_v(\mu_{p-1})$ and $L_v^{ur}$ are linear disjoint, the cyclotomic character $\varepsilon: I_{L,v}\to \Z_p^\times\to \F_p^\times$ is surjective. If $k>2$,
then
\[(\Tfn^-)^{I_{L,v}}=(\cOn(\e^{1-\frac{k}{2}}))^{I_{L_v}}={0}\Longrightarrow (\Tfn^-)^{G_{L_v}}=\{ 0 \} .\]
If $k=2$,  then the Frobenius $\Frob_v$ acts on $\Tfn^-$ by a scalar
$\al_p(f)^{r_v}$, where $\al_p(f)$ is the unit root of the Hecke polynomial $X^2-\bfa_p(f)X+p^{k-1}$
and $r_v=1$ if $p$ is split and $r_v=2$ if $p$ is inert. By \eqref{PO}, $\al_p(f)^2-1$ is a $p$-adic unit. Then we find that
\[(\Tfn^-)^{G_{L_v}}=(\Tfn^-)[(\mathrm{Frob}_v)^{p^s}-1]=(\Tfn^-)[\al_p(f)^{r_v p^s}-1]=\{ 0 \}.\]
Here $p^s$ is the inertia degree of $L_v/K_v$. This completes the proof.
\end{proof}
\subsection{Global reciprocity}
By the global reciprocity law of class field theory, for $\kappa \in \widehat{H}^1(K_\infty,\Tfn)$ and $s\in H^1(K_\infty,\Afn)$, we have
$$
\sum_{q\textup{: prime}} \langle \res_q(\kappa) , \res_q(s)  \rangle_q =0.
$$
Since the local conditions of $\wh\Sel_\Delta^S(K_\infty,\Tfn)$ and $\Sel_\Delta (K_\infty,\Afn)$
are orthogonal at the primes not dividing $S$,
if $\kappa$ belongs to $\wh\Sel_\Delta^S(K_\infty,\Tfn)$ and $s$ belongs to $\Sel_\Delta (K_\infty,\Afn)$,
then we have
\beq\label{E:globalRec}
\sum_{q\divides S} \langle \partial_q (\kappa ), v_q(s)  \rangle_q =0.\eeq

\subsection{Control theorem (I)}
Let $S$ be a square-free product of $n$-admissible primes. We prove control theorems for the discrete Selmer groups $\Sel^S_\Delta(K_\infty,\Afn)$. Let $L/K$ be a finite extension in $K_\infty$.
\begin{prop}\label{P:control1}With \CR\,and  \eqref{PO}, assume further that
\beqcd{Irr} \text{the residual Galois representation $\bar\rho_f:G_\Q \to \Aut_\cO(A_{f,1})$ is absolutely irreducible.}\eeqcd
Then
\begin{mylist}
\item the restriction maps \[H^1(K,\Afn)\to H^1(L,\Afn)^{\Gal(L/K)},\quad\Sel_\Delta^S(K,\Afn)\to \Sel_\Delta^S(L,\Afn)^{\Gal(L/K)}\] are isomorphisms.
\item  $H^1(K,\Afn)=H^1(L,A_f)[\uf^n]$ and $\Sel_\Delta^S(L,\Afn)=\Sel_\Delta^S(L,A_f)[\uf^n]$.
\end{mylist}
\end{prop}
\begin{proof}
Since $K_\infty/\Q$ is a Galois extension, by the residual irreducibility of $\overline{\rho}_{f}$ we have $\Tfn^{G_{K_\infty}}=0$, and hence we have an isomorphism $H^1(K,\Afn)\iso H^1(L,\Afn)^{\Gal(L/K)}$. To show part (1), it suffices to show
\begin{enumerate}
\item[(a)] $H^1(K_\ell^{ur},\Afn)\to H^1(L_\ell^{ur},\Afn)$ is injective for $\ell\ndivides p\Delta$, and
\item[(b)]  $H^1(K_\ell,\Afn/F_\ell^+\Afn)\to H^1(L_\ell,\Afn/F_\ell^+\Afn)$ is injective for $\ell\divides p\Delta$.
\end{enumerate}
Since $L/K$ is unramified outside $p$ and anticyclotomic, for each place $\lam$ of $L$ above $\ell\not =p$, we have $K_\lam^{ur}=L_\lam^{ur}$. This verifies (a). Part (b) for $\ell\divides \Delta$ follows from the fact that $K_\lam=L_\lam$ for any prime $\lam$ of $L$ above $\ell$ which is non-split in $K$. Part (b) for $\ell=p$ is an immediate consequence of \lmref{L:5}.

Part (2)  can be proved in the same way. In view of the exact sequence
\[0\to \Afn\to A_f\stackrel{\uf^n}\to A_f\to 0\]
and the fact that $\Afn^{G_L}\subset \Afn^{G_{K_\infty}}=0$, we have $H^1(L,\Afn)=H^1(L,A_f)[\uf^n]$, and hence the natural map $\Sel_\Delta^S(L,\Afn)\hookto\Sel_\Delta^S(L,A_f)$ is injective. To show the surjectivity, it remains to show
\begin{enumerate}
\item[(a1)] $H^1(L_\ell^{ur},\Afn)\to H^1(L_\ell^{ur},A_f)$ is injective for $\ell\ndivides p\Delta$, and
\item[(b1)]  $H^1(L_\ell,\Afn/F_\ell^+\Afn)\to H^1(L_\ell,A_f/F_\ell^+A_f)$ is injective for $\ell\divides p\Delta$.
\end{enumerate}
The injectivity of (a1) and (b1) for $\ell\not =p$ is deduced from the fact that $A_f^{I_{L,\ell}}$ and $(A_f/F_\ell^+A_f)^{G_{L_\ell}}$ are $p$-divisible under the hypothesis \CR. The case (b1) for $\ell=p$ follows from \lmref{L:5}.
\end{proof}


\def\TBU{\bbT_B(N^+,p^n)}
\newcommand\Hecke[3]{\bbT_{#1}(#2,#3)}



\section{\padic modular forms on definite quaternion algebras}\label{S:modularForms}
\subsection{Hecke algebras of quaternion algebras}\label{SS:Heckealgebra}
Let $B$ be a quaternion algebra over $\Q$ and let $\Bhat=B\ot_\Z\Zhat$ be the profinite completion of $B$. Denote by $\Delta_B$ the absolute discriminant of $B$. If $\Sg$ is a positive integer, then $\wh B^\x=\wh B^{(\Sg)\x}\wh B_{(\Sg)}^\x$, where
\[\Bhat^{(\Sg)\x}=\stt{x\in \Bhat^\x\mid x_\pme =1\,\forall \pme\divides \Sg}\,;\,\wh B_{(\Sg)}^\x=\stt{x\in \wh B^\x\mid x_\pme=1\,\forall \pme\ndivides\Sg}.\]
If $\cU$ is an open compact subgroup of $\Bhat^\x$ and $x\in \wh B^\x$, we denote by $\cU_\pme$ (resp. $x_\pme$) the local component of $\cU$ (resp. $b$) for each prime $\pme$.

The Hecke algebra $\cH(B^\x,\cU)$ is the space of $\Z$-valued bi-$\cU$-invariant functions on $\Bhat^\x$ with compact support, equipped with the convolution product
\[(\al*\beta)(x)=\int_{\Bhat^\x}\al(y)\beta(y^{-1}x)dy,\]
where $dy$ is the Haar measure on $\Bhat^\x$ normalized so that $\vol(\cU,d y)=1$. For each $x\in\Bhat^\x$, denote by $[\cU x\cU]$ the characteristic function of the set $\cU x\cU$.
For each prime $\pme$, we set
\beq\label{E:1}\pi_\pme=\pDII{\pme}{1},\,z_\pme=\pDII{\pme}{\pme}\in\GL_2(\Q_\pme).\eeq

We fix an identification $i:\wh B^{(\Delta_B)}\iso M_2(\wh\Q^{(\Delta_B)})$. Let $M^+$ be a positive integer with $(M^+,\Delta_B)=1$ and let $R_{M^+}$ be the Eichler order in $B$ of level $M^+$ with respect to $i$. Suppose that $p\ndivides \Delta_B$. For a non-negative integer $n$, we define a special open compact subgroup $\cU_{M^+,p^n}\subset \wh B^\x$ by
\beq\label{E:opncmpt}\cU_{M^+,p^n}:=\stt{x\in \wh R_{M^+}^\x\mid x_p\con \pMX{a}{b}{0}{a}\pmod{p^n},\,a,b\in\Zp}.\eeq
Let $\cU=\cU_{M^+,p^n}$. Let $\Sg:=p^nM^+\Delta_B$. By definition, we have $i:(\wh B^{(\Sg)\x},\cU^{(\Sg)})\iso (\GL_2(\wh\Q^{(\Sg)}),\GL_2(\Zhat^{(\Sg)}))$. Denote by $\bbT_B^{(\Sg)}(M^+)$ the subalgebra of $\cH(B^\x,\cU)$ generated by
$[\cU x\cU]$ with $x\in\wh B^{(\Sg)\x}$. Then
\[\bbT_B^{(\Sg)}(M^+)=\Z[\,T_\pme,S_\pme,S_\pme^{-1}\mid \pme\ndivides\Sg\,],\]
where $T_\pme$ and $S_\pme$ are the standard Hecke operators at the prime $\pme$ given by
\beq\label{E:2}T_\pme=[\cU \,i^{-1}(\pi_\pme)\,\cU],\quad  S_\pme=[\cU\,i^{-1}(z_\pme)\,\cU].\eeq
We proceed to introduce Hecke operators at $\pme\divides\Sg$. If $\pme\divides\Delta_B$, choose $\pi_\pme'\in B^\x_\pme$ with $\rmN(\pi_\pme')=\pme$. Define the Hecke operator $U_\pme$ at $\pme\divides M^+\Delta_B$ by
\beq\label{E:Upoperators}U_\pme=[\cU\,i^{-1}(\pi_\pme)\,\cU]\text{ if }\pme\divides M^+,\quad U_\pme=[\cU\pi'_\pme \cU]\text{ if }\pme\divides\Delta_B.\eeq
If $n=0$, then $\cU=\wh R_{M^+}^\x$, and we define
\beq\label{E:HeckealgebraI}\bbT_B(M^+)=\Z\left[\stt{T_\pme,S_\pme,S^{-1}_\pme\mid \pme\ndivides\Sg},\stt{U_\pme\mid \pme|\Sg}\right]=\bbT_B^{(\Sg)}(M^+)\left[\stt{U_\pme\mid \pme\divides\Sg}\right].\eeq
Suppose that $n>0$. For the prime $p$, we define Hecke operators $U_p$ and $\Dmd{a}$ for each $a\in\Zp^\x$ by
\beq\label{E:Diamondop}U_p=[\cU\,i^{-1}(\pi_p)\,\cU],\quad \Dmd{a}=[\cU\,i^{-1}(\bfd(a))\,\cU]\quad(\bfd(a)=\pDII{a}{1}\in\GL_2(\Zp)).\eeq
Let $\Hecke{B}{M^+}{p^n}$ be the (commutative) subring of $\cH(B^\x,\cU)$ generated by Hecke operators at all primes. Namely,
\beq\label{E:Heckealgebra}\begin{aligned}\Hecke{B}{M^+}{p^n}=&\Z\left[\stt{T_\pme,S_\pme,S^{-1}_\pme\mid \pme\ndivides\Sg},\stt{U_\pme\mid \pme\divides M^+\Delta_B},\stt{U_p,\Dmd{a}\mid a\in\Zp^\x}\right]\\
=&\bbT_B^{(\Sg)}(M^+)\left[\stt{U_\pme,\Dmd{a}\mid \pme\divides\Sg,\,a\in\Zp^\x}\right].\end{aligned}\eeq
We call $\Hecke{B}{M^+}{p^n}$ ($\bbT_B(M^+)$ if $n=0$) the complete Hecke algebra of level $\cU$. For each prime $\ell$, let $\bbT_B^\setl(M^+,p^n)$ be the subring of $\Hecke{B}{M^+}{p^n}$ generated by Hecke operators at primes $\pme\not =\ell$.
\subsection{Definite quaternion algebras}
We recall some notation from \cite[\S2.1]{Hsieh_Chida}. Let $\cK$ be the imaginary quadratic field with the discriminant $-D_\cK<0$ and let $\Diff=\sqrt{-D_\cK}$. Write $z\mapsto \ol{z}$ for the complex conjugation on $\cK$. Define $\CMP\in\cK$ by
\[\CMP=\frac{D'+\delta}{2},\,D'=\begin{cases}D_\cK&\text{ if }2\ndivides D_\cK,\\
D_\cK/2&\text{ if } 2\divides D_\cK.
\end{cases}\]
Then $\cO_\cK=\Z+\Z\cdot\CMP$ and $\CMP\ol{\CMP}$ is a local uniformizer of primes that are ramified in $\cK$.

Let $B$ be the definite quaternion algebra over $\Q$ with discriminant $\Delta_B$. Suppose that
\beq\label{E:assumption}N^-\mid \Delta_B,\,(\Delta_B,N^+)=1\text{ and }p\ndivides \Delta_B.\eeq Assume further that every prime factor of $\Delta_B/N^-$ is inert in $\cK$. Thus, we can regard $\cK$ as a subalgebra of $B$.
Write $\rmT$ and $\rmN$ for the reduced trace and norm of $B$ respectively. Let $\frakp$ be the prime of $\cK$ above $p$ induced by $\iota_p:\cK\hookto\Cp$. We choose a basis of $B=\cK\oplus \cK\cdot\cmJ$ over $\cK$ such that
\begin{itemize}
\item[(a)] $\cmJ^2=\beta\in\Q^\x$ with a square-free $\beta<0$ and $\cmJ t=\ol{t}\cmJ$ for all $t\in\cK$,
\item[(b)] $\beta\in(\Z_\pme^\x)^2$ for all $\pme\divides N^+$ and $\beta\in\Z_\pme^\x$ for $\pme\divides D_\cK$.
\end{itemize}
The existence of such $\cmJ$ can be seen as follows. We can always choose some $J'\in B^\x$ satisfying (a) by Noether-Skolem theorem, and the strong approximation theorem ensures the existence of $\cmJ=\beta J'$ with the property (b) for some suitable $\beta\in\cK^\x$.

Fixing a square root $\sqrt{\beta}\in\Qbar$ of $\beta$, we require the fixed isomorphism $i=\prod_{\pme\ndivides \Delta_B}i_\pme:\wh B^{(\Delta_B)}\iso M_2(\wh\Q^{(\Delta_B)})$ is chosen so that for each finite place $\pme |pN^+$, the isomorphism $i_\pme:B_\pme=B\ot_\Q\Qq\iso M_2(\Qq)$ is given by
\beq\label{E:embedding.W}i_\pme(\bftheta)=\pMX{\rmT(\bftheta)}{-\rmN(\bftheta)}{1}{0};\quad i_\pme(\cmJ)=\sqrt{\beta}\cdot \pMX{-1}{\rmT(\bftheta)}{0}{1}\quad(\sqrt{\beta}\in\Z_\pme^\x),\eeq
and for each finite place $\pme\ndivides p N^+\Delta_B$, $i_\pme:B_\pme\iso M_2(\Qq)$ satisfies
 \beq\label{E:21.W}i_{\pme}(\cO_\cK\ot\Z_\pme)\subset M_2(\Z_\pme).\eeq
Hereafter, we shall identify $\wh B^{(\Delta_B)\x}$ with $M_2(\wh\Q^{(\Delta_B)})$ via $i$ and let \[\cU=\cU_{N^+,p^n}\] be the open compact subgroup as in \eqref{E:opncmpt}. By definition,
\beq\label{E:3}\cU\supset \Zhat^\x\,;\quad\cU_\ell=\GL_2(\Z_\ell)\text{ if }\ell\ndivides pN^+\Delta_B.\eeq
\subsection{\padic modular forms}
Let $A$ be a \padic ring. Let $k\geq 2$ be an even integer and let $L_k(A)=A[X,Y]_{k-2}$ be the space of homogeneous polynomials of degree $k-2$ over $A$. Let $\rho_k:\GL_2(A)\to \Aut_AL_k(A)$ be the unitary representation defined by
 \[\rho_k(h)P(X,Y)=\det(h)^{-\frac{k-2}{2}}\cdot P((X,Y)h)\quad( h\in\GL_2(A),\,P(X,Y)\in L_k(A)).\]
Define the space $\SBk(\cU,A)$ of \padic modular forms on $B^\x$ of weight $k$ and level $\cU$ by
\[\SBk(\cU,A)=\stt{f:B^\x\bksl \wh B^\x\to L_{k}(A)\mid f(bu)=\rho_k(u_p^{-1})f(b),\,u=(u_\pme)\in \cU}.\]
The space $\SBk(\cU,A)$ is equipped with $\Hecke{B}{N^+}{p^n}$-action defined by
\beq\begin{aligned}\label{E:HeckeFormulae} [\cU x\cU].f(b)=&\sum_{u\in \cU/\cU\cap x\cU x^{-1}}f(bux)\text{ if }x\in \wh B^{\setp\x};\\
\Dmd{a}.f(b)=&\rho_k(\bfd(a))f(b\bfd(a))\quad(a\in\Zp^\x);\\
U_p.f(b)=&\sum_{u\in \cU_p/\cU_p\cap \pi_p\cU_p\pi_p^{-1}}\rho_k(u)\wtd\rho_k(\pi_p)f(bu\pi_p)\quad (\pi_p=\pDII{p}{1}\in\GL_2(\Qp)),\end{aligned}\eeq
where $\wtd\rho_k(\pi_p)\in\End_A L_k(A)$ is defined by $\wtd\rho_k(\pi_p)P(X,Y):=P(pX,Y)$. If $p$ is invertible in $A$, then
\[U_p.f(b)=p^\frac{k-2}{2}\cdot \sum_{u\in \cU_p/\cU_p\cap \pi_p\cU_p\pi_p^{-1}}\rho_k(u\pi_p)f(bu\pi_p).\]

We note that the operator $S_\pme$ acts trivially and $U_\pme$ is an involution if $\pme\divides\Delta_B$ on $\SBk(\cU,A)$.
\subsection{}\label{SS:weighttwo}
We shall write $\SB(\cU,A)=\SB_2(\cU,A)$ for the space of $A$-valued modular forms of weight two and level $\cU$. Denote by $\XB(\cU)$ the finite set
\[\XB(\cU):=B^\x\bksl \wh B^\x/\cU.\]
For each $b\in\wh B^\x$, denote by $[b]_\cU$ the point in $\XB(\cU)$ represented by $b$. The Hecke algebra $\Hecke{B}{N^+}{p^n}$ acts on the divisor group $\Z[\XB(\cU)]$ by Picard functoriality. By definition, we have a canonical identification:
 \beq\label{E:weighttwo}\SB(\cU,\Z)=\stt{f:\XB(\cU)\to \Z}\iso \Z[\XB(\cU)].\eeq

Define the \emph{Atkin-Lehner involution} $\tau_n\in \wh B_{(pN^+)}^\x$ by $\tau_{n,\pme}=\pMX{0}{1}{p^nN^+}{0}$ if $\pme\divides p^nN^+$ 
 and $\tau_{n,\pme}=1$ if $\pme\ndivides pN^+$. Then $\tau_n$ normalizes $\cU$, and hence induces a right action on $\XB(\cU)$. Define a perfect pairing $\pairing_{\cU}:\SB(\cU,A)\x \SB(\cU,A)\to A$ by
\beq\label{E:pairing}\pair{f_1}{f_2}_\cU=\sum_{[b]_\cU\in \XB(\cU)}f_1(b)f_2(b\tau_n)\cdot\#((B^\x\cap b\cU b^{-1})/\Q^\x)^{-1}.\eeq
Then the action of $\TBU$ on $\cS^B(\cU,A)$ is self-adjoint with respect to $\pairing_\cU$. Namely,
\[\pair{tf_1}{f_2}_\cU=\pair{f_1}{tf_2}_\cU\text{ for all }t\in\TBU.\]

\section{Shimura curves}\label{S:ShimuraCurve}
\subsection{}\label{SS:ShimuraCurve}
 We recall some basic facts of geometry of Shimura curves, following the exposition in \cite[\S1]{Nekovar:Level_raising}. In this section, let $\ell\ndivides N^+\Delta_B$ be a rational prime which is \emph{inert} in $\cK$ and $B'$ be the indefinite quaternion algebra over $\Q$ with discriminant $\Delta_B\ell$. We fix a $\Q$-embedding $t':\cK\hookto B'$ and an isomorphism $\vp_{B,B'}:\wh B^\setl\iso \wh B^{\prime\setl}$ once and for all such that $t'$ induces the composite map
\[\wh\cK^\setl\to \wh B^\setl \stackrel{\vp_{B,B'}}\longto \wh B^{\prime\setl}.\]
We put
\[\cU_0(\ell):=\cU_{N^+\ell,p^n}=\stt{u\in \cU\mid u_\ell\con \pMX{*}{*}{0}{*}\pmod{\ell}}.\]
Let $\cO_{B'_\ell}$ be the maximal order of $B'_\ell$. Let $\cU'$ be the open compact subgroup of $\wh B^{\prime \x}$ given by \beq\label{E:opncpmt'}\cU':=\vp_{B,B'}(\cU^\setl)\cO_{B'_\ell}^\x.\eeq
We denote by $M_{\cU'}{}_{/\Q}$ the Shimura curve attached to $B'$ of level $\cU'$. The complex uniformization of $M_{\cU'}$ is given by
\[M_{\cU'}(\C)=B^{\prime\x}\bksl (\C\setminus\R)\xx \wh B^{\prime\x} /\cU'.\]
For $z\in \C\setminus\R$ and $b'\in \wh B^{\prime\x}$, denote by $[z,b']_{\cU'}$ the point of $M_{\cU'}(\C)$ represented by $(z,b')$.

\subsection{$\ell$-adic uniformization of Shimura curves}\label{S:ladicuniform}
There is an integral model $\cM_{\cU'}$ of $M_{\cU'}\ot_\Q\Q_\ell$ over $\Z_\ell$, which is projective over $\Z_\ell$ if $\cU'$ is sufficiently small \cite[Theorem\,(3.4)]{Boutot_Carayol:Drinfeld}. We review the description of the completion $\wh\cM_{\cU'}$ of $\cM_{\cU'}$ along the special fibre due to \v{C}erednik \cite{Cerednik} and Drinfeld \cite{Drinfeld:CoveringPadic}. Let $\cH_\ell$ be Drinfeld's $\ell$-adic upper half plane, which is a rigid analytic variety over $\Q_\ell$ and an analytic subspace of $\bfP_{\Q_\ell}^{1,\an}$. The $\C_\ell$-valued points of $\cH_\ell$ are $\cH_\ell(\C_\ell)=\bfP^1(\C_\ell)\setminus \bfP^1(\Q_\ell)$. Then $\wh\cM_{\cU'}$ is canonically identified with
\beq\label{E:specialFibre}B^\x\bksl \wh\cH_\ell\wh\ot_{\Z_\ell}\Zhat_\ell^{ur}\xx\wh B^{\setl\x}/\cU^\setl,\eeq
where $\wh\cH_\ell$ is a natural formal model of $\cH_\ell$ and $b\in B^\x$ acts on $\wh\cH_\ell$ (resp. on $\Zhat_\ell^{ur}$) via the natural action of $B^\x\subset B_\ell^\x\stackrel{i_\ell}\iso\GL_2(\Q_\ell)$ on $\bfP^{1,\an}_{\Q_\ell}$ (resp. by $\Frob_\ell^{\Ord_\ell(\rmN(b))}$) (\cite[Theorem\,5.2]{Boutot_Carayol:Drinfeld}). Denote by $M_{\cU'}^{\an}$ the rigid analytification of $M_{\cU'}\ot\Q_\ell$. Then $M_{\cU'}^{\an}\iso\wh\cM_{\cU'}\ot\Q_\ell$ the generic fibre of $\wh\cM_{\cU'}$, and
\[M_{\cU'}^{\an}=B^\x\bksl \cH_\ell\wh\ot_{\Q_\ell}\wh\Q_\ell^{ur}\xx \wh B^{\setl\x}/\cU^\setl.\]
Since $\cU^\setl\supset \Zhat^{\setl \x}$,
\beq\label{E:Clpoints}M_{\cU'}(\C_\ell)=B^\x\bksl\cH_\ell(\C_\ell)\xx \wh B^{\setl \x}/\cU^\setl.\eeq
\subsection{Bad reduction of Shimura curves}
Let $\sT_\ell=\cV(\sT_{\ell})\disjoint \cE(\sT_\ell)$ be the dual graph of the special fibre of $\wh\cH_\ell$, where $\cV(\sT_\ell)$ and $\cE(\sT_\ell)$ denote the set of vertices and edges of $\sT_\ell$ respectively. Then $\sT_\ell$ is the Bruhat-Tits tree of $B_\ell^\x\iso\GL_2(\Q_\ell)$. Let $\overrightarrow{\cE}(\sT_\ell)$ be the set of oriented edges. We have the identifications:
\[\cV(\sT_\ell)=B_\ell^\x/\cU_\ell\Q_\ell^\x,\,\overrightarrow{\cE}(\sT_\ell)=B^\x_\ell/\cU_0(\ell)_\ell\Q_\ell^\x.\]

Let $\cG=\cV(\cG)\disjoint\cE(\cG)$ be the dual graph of the special fibre of $\wh\cM_{\cU'}\ot_{\Z_\ell}\cO_{\cK_\ell}$. The set $\cV(\cG)$ of vertices of $\cG$ consists of the irreducible components of $\cM_{\cU'}\ot_{\Z_\ell}\F_{\ell^2}$, and the set $\cE(\cG)$ of edges of $\cG$ consists of the singular points in $\wh\cM_{\cU'}\ot_{\Z_\ell} \F_{\ell^2}$. Let $\red_\ell:M_{\cU'}(\C_\ell)\to \cG=\cV(\cG)\disjoint\cE(\cG)$ be the reduction map. By \cite[Proposition\,1.5.5]{Nekovar:Level_raising}, \eqref{E:specialFibre} induces an identification:
\beq\label{E:vertices}\begin{aligned}\cV(\cG)=B^\x\bksl\left( \cV(\sT_\ell)\xx\Z/2\Z\xx \wh B^{\setl\x}/\cU^\setl\right)&=B^\x\bksl\left(B_\ell^\x/U_\ell\Q_\ell^\x\xx\Z/2\Z\xx \wh B^{\setl\x}/\cU^\setl\right)\\
&\isoto \left(B^\x\bksl \wh B^\x /\cU\right)\xx\Z/2\Z=\XB(\cU)\xx \Z/2\Z,\end{aligned}\eeq
where the last isomorphism is given by
\[B^\x\left(b_\ell \cU_\ell,j,b^\setl \cU^\setl\right)\mapsto \left([b_\ell b^\setl]_\cU,j+\Ord_\ell(\rmN(b_\ell))\right)\]
and
\[\overrightarrow{\cE}(\cG)\isoto \left(B^\x\bksl \wh B^\x/\cU_0(\ell)\right)\xx\Z/2\Z=\XB(\cU_0(\ell))\xx \Z/2\Z.\]
We will regard $\cE(\cG)$ as a subset of $\overrightarrow{\cE}(\cG)$  via
\beq\label{E:edges}\cE(\cG)=\XB(\cU_0(\ell))\isoto \XB(\cU_0(\ell))\x\stt{0}\hookto \overrightarrow{\cE}(\cG). \eeq


\subsection{Bad reduction of the Jacobian of the Shimura curve}\label{SS:badreduction}
Let $J(M_{\cU'})$ be the Jacobian of the Shimura curve $M_{\cU'}$. If $L/\Q$ is a field extension, let $\Div^0 M_{\cU'}(L)$ be the group of divisors on $J(M_{\cU'})(L)$ of degree zero on each connected component of $M_{\cU'}\ot_\Q\ol{L}$. For $D\in\Div^0 M_{\cU'}(L)$, denote by $cl(D)\in J(M_{\cU'})(L)$ the point represented by $D$. The prime-to-$\ell$ Hecke algebra $\cH^\setl(B^{\prime\x},\cU')$ acts on $J(M_{\cU'})$ via the Hecke correspondence on $M_{\cU'}$ and Picard functoriality (\cf\cite[\S1.3.4]{Nekovar:Level_raising}).
The isomorphism $\vp_{B,B'}:\wh B^{\setl}\iso \wh B^{\prime\setl}$ induces an isomorphism
\[\vp_*:\cH^\setl(B^\x,\cU_0(\ell))\isoto \cH^\setl(B^{\prime\x},\cU'),\quad [\cU_0(\ell) x\cU_0(\ell)]\mapsto [\cU'\vp_{B,B'}(x)\cU']\quad(x\in B^{\setl\x}).\]
We extend $\vp_*$ to a ring homomorphism
\[\vp_*:\Hecke{B}{\ell N^+}{p^n}\surjto\Hecke{B'}{N^+}{p^n}\to\End(J(M_{\cU'})_{/\Q})\]
by defining $\vp_*(U_\ell):=[\cU'\pi'_\ell\cU']$ for some $\pi_\ell'\in B'_\ell$ with $\rmN(\pi'_\ell)=\ell$.

Let $\cJ$ be the \Neron model of $J(M_{\cU'})_{/\Q_\ell}$ over $\Z_\ell$. The universal property of \Neron models induces a ring homomorphism
\beq\label{E:4}\vp_*:\Hecke{B}{\ell N^+}{p^n}\to\End(J(M_{\cU'})_{/\Q_\ell})=\End(\cJ).\eeq
Let $\cJ_s$ be the special fibre $\cJ\ot_{\Z_\ell}\F_{\ell^2}$ and $\cJ_s^\circ$ be the connected component of the identity of $\cJ_s$.  Let $\Phi_{M_{\cU'}}=\cJ_s/\cJ_s^\circ$ be the group of connected components of $\cJ_s$. Then $\Phi_{M_{\cU'}}$ is an \etale group scheme over $\F_{\ell^2}$ with a natural $\Hecke{B}{\ell N^+}{p^n}$-module structure induced by \eqref{E:4}. Let
\[r_\ell:J(M_{\cU'})(\cK_\ell)\to\Phi_{M_{\cU'}} \]
be the reduction map.

We recall a description of $\Phi_{M_{\cU'}}$ in terms of the graph $\cG$. Define the source and target maps $s,t:\overrightarrow{\cE}(\cG)\to \cV(\cG)$ so that for each oriented edge $e$, $s(e)\in\cV(\cG)$ is the source of $e$ and $t(e)\in \cV(\cG)$ is the target of $e$.
Define the morphisms
\beq\label{E:d*}\begin{CD}\Z[\cE(\cG)]@>d_*=-s_*+t_*>>\Z[\cV(\cG)],\quad \Z[\cV(\cG)]@>d^*=-s^*+t^*>> \Z[\cE(\cG)].\end{CD}\eeq
Let $\Z[\cV(\cG)]_0$ be the image of $d_*$ in $\Z[\cV(\cG)]$.
By \cite[\S9.6,\,Theorem\,1]{NeronModel}, we have a canonical isomorphism
\beq\label{E:components}\Z[\cE(\cG)]/\Im d^*\stackrel{d_*}\isoto \Z[\cV(\cG)]_0/\Im d_*d^*\iso \Phi_{M_{\cU'}}\eeq
such that the following diagram commutes:
\beq\label{E:comutativediagram}\begin{CD}\Div^0 M_{\cU'}(\cK_\ell)@>cl>>J(M_{\cU'})(K_\ell)\\
@VVr_\cV V  @VVr_\ell V\\
\Z[\cV(\cG)]_0/\Im d_*d^* @>\sim>> \Phi_{M_{\cU'}},
\end{CD}\eeq
where $r_\cV:\Div^0 M_{\cU'}(\cK_\ell)\to \Z[\cV(\cG)]_0$ is the specialization map of divisors defined in \cite[\S1.6.6]{Nekovar:Level_raising}. We briefly recall the definition of $r_\cV$ as follows. For each $D\in\Div^0 M_{\cU'}(\cK_\ell)$, extending $D$ to a Cartier divisor $\wtd D$ on $\cM_{\cU'}\ot_{\Z_\ell}\cO_{\cK_\ell}$ by taking closure, define
\beq\label{E:spDivisors}r_\cV(D)=\sum_{C\in \cV(\cG)}(\wtd D\cdot C)C\in\Z[\cV(\cG)]_0,\eeq
where $(\wtd D\cdot C)$ is the intersection number in $\cM_{\cU'}\ot_{\Z_\ell}\cO_{\cK_\ell}$.

\subsection{}We review a description of $\Phi_{M_{\cU'}}$ in terms of spaces of weight two modular forms (\cite[\S5.5]{Bertolini_Darmon:IMC_anti}). Let $\al,\beta:\XB(\cU_0(\ell))\to\XB(\cU)$ be the standard degeneracy maps given by
\[x=[b]_{\cU_0(\ell)}\mapsto \al(x)=[b]_{\cU},\,\beta(x)=[b\pi_\ell^{-1}]_{\cU}.\]
According to \eqref{E:vertices} and \eqref{E:edges}, the morphisms $d_*$ and $d^*$ in \eqref{E:d*} are respectively identified with
\[\begin{CD}\SB(\cU_0(\ell),\Z)@>\delta_*=(-\al_*,\beta_*)>>\SB(\cU,\Z)^{\oplus 2},\quad \SB(\cU,\Z)^{\oplus 2}@>\delta^*=-\al^*+\beta^*>> \SB(\cU_0(\ell),\Z).\end{CD}\]
Let $(\SB(\cU,\Z)^{\oplus 2})_0:=\Im\delta_*$ be the image of $\SB(\cU_0(\ell),\Z)$ via $\delta_*$.
A direct computation shows that
\[\delta_*\delta^*=(-\al_*,\beta_*)(\al^*-\beta^*)=\pMX{-\ell-1}{T_\ell}{T_\ell}{-\ell-1}\in M_2(\End(\SB(\cU,\Z))).\]
Define a ring homomorphism $\Hecke{B}{\ell N^+}{p^n}\to\End (\SB(\cU,\Z)^{\oplus 2})$ by
\begin{align*}
t\to&\wtd t:(x,y)\mapsto (tx,ty)\text{ if }t\in\bbT_B^\setl(\ell N^+,p^n)=\bbT_B^\setl(N^+,p^n);\\
U_\ell\to &\wtd U_\ell:(x,y)\mapsto (-\ell y,x+T_\ell y).
\end{align*}
This makes $\SB(\cU,\Z)^{\oplus 2}$ a $\Hecke{B}{\ell N^+}{p^n}$-module. Moreover,
one can check that $\delta_*$ is indeed a $\Hecke{B}{\ell N^+}{p^n}$-module homomorphism.

\begin{prop}[Proposition\,5.13,\,\cite{Bertolini_Darmon:IMC_anti}]\label{P:1}We have an isomorphism as $\Hecke{B}{\ell N^+}{p^n}$-modules
\[(\SB(\cU,\Z)^{\oplus 2})_0/(\wtd U_\ell^2-1)\SB(\cU,\Z)^{\oplus 2}\iso \Phi_{M_{\cU'}}.\]\end{prop}
\begin{proof}A direct computation shows that
\[\wtd U_\ell^2-1=\delta_*\delta^*\circ \tau;\quad \tau(x,y)=(x+T_\ell y,y).\]
Since $\tau$ is an automorphism of $\SB(\cU,\Z)^{\oplus 2}$, we can deduce the proposition from the identification between $\Z[\cE(\cG)]$ and $\SB(\cU_0(\ell),\Z)$ as $\Hecke{B}{\ell N^+}{p^n}$-modules combined with the canonical isomorphism \eqref{E:components} and the compatibility of Hecke actions \cite[pp.463--464]{Ribet1990:SerreEpsilonConjecture}(\cf\cite[Proposition\,5.8]{Bertolini_Darmon:IMC_anti} and \cite[\S1.6.7]{Nekovar:Level_raising}).
\end{proof}

\subsection{CM points in Shimura curves}
Take a point $z'$ in $\C\setminus\R$ fixed by $i_\infty(\cK^\x)\subset \GL_2(\R)$. The set of CM points by $\cK$ unramified at $\ell$ on the curve $M_{\cU'}$ is defined as
\[\CM^{\ell-ur}_\cK(M_{\cU'}):=\stt{[z',b']_{\cU'}\mid b'\in\wh B^{\prime\x},\,b'_\ell=1}\subset M_{\cU'}(\cK^{ab}).\]
 Let $\rec_\cK:\wh \cK^\x\to \Gal(\cK^{ab}/\cK)$ be the geometrically normalized reciprocity law. Then Shimura's reciprocity law says
\beq\label{E:ShimuraRec}\rec_\cK(a)[z',b']_{\cU'}=[z',t'(a)b']_{\cU'}.\eeq
This implies
\[\iota_\ell:\CM^{\ell-ur}_\cK(M_{\cU'})\hookto M_{\cU'}(\cK_\ell).\]
Let $\CM^{\ell-ur}_\cK(M_{\cU'})^0$ be the subgroup of $\Div^0 M_{\cU'}(\cK_\ell)$ generated by the degree zero divisors supported in $\CM_\cK^{\ell-ur}(M_{\cU'})$. Then the specialization map $r_\cV:\CM^{\ell-ur}_\cK(M_{\cU'})^0\to \Z[\cV(\cG)]_0$ is given by
\beq\label{E:specialization}r_\cV(\sum_i n_i\cdot [z',b_i']_{\cU'})=\sum_i n_i\cdot [\vp_{B,B'}^{-1}(b_i')]_\cU.\eeq

\section{Construction of the Euler system}\label{S:EulerSystem}
\subsection{The set-up}
Let $f\in S_k(\Gamma_0(N))$ be an elliptic new form of level $N$ with $q$-expansion at the infinity cusp
\[f(q)=\sum_{n>0}\bfa_n(f)q^n.\]
Let $\Q(f)$ be the Hecke field of $f$, \ie the finite extension of $\Q$ generated by $\stt{\bfa_n(f)}_{n}$. Let $\cO$ be a finite extension of $\Zp$ containing the ring of integers of $\Q(f)$. Then it is well-known that $\bfa_\pme(f)$ belongs to $\cO$. We set
\[\al_p(f)=\text{ the \padic unit root of $X^2-\bfa_p(f)X+p^{k-1}$ in $\Cp$},\quad \al_\pme(f):=\bfa_\pme(f)\pme^\frac{2-k}{2}\text{ if }\pme\not =p.\]
We define an $\cO$-algebra homomorphism
\[\lam_f:\bbT_B(N^+,p)_\cO=\bbT_B(N^+,p)\ot_\Z\cO\to \cO\] by $\lam_f(T_\pme)=\al_\pme(f),\,\lam_f(S_\pme)=1$ if $\pme\ndivides pN$ and $\lam_f(U_\pme)=\al_\pme(f)$ if $\pme\divides pN$, $\lam_f(\Dmd{a})=a^\frac{k-2}{2}$ for $a\in\Zp^\x$.




\subsection{Level raising}Let $n$ be a positive integer and let $\cOn=\cO/(\uf^n)$.
Recall that we have introduced the notion of $n$-admissible primes for $f$ in \defref{D:admissibleprimes}.
\begin{defn}\label{D:1}
An \emph{$n$-admissible form} $\cD=(\Delta,\fn)$ is a pair consisting of a square-free integer $\Delta$ of an odd number of prime factors and a \padic quaternionic eigenform $\fn\in \SB(\cU_{N^+,p^n},\cOn)$ for the definite quaternion algebra $B$ over $\Q$ of discriminant $\Delta$ such that the following conditions hold:
\begin{mylist}
\item $N^-\mid \Delta$ and every prime factor of $\Delta/N^-$ is $n$-admissible.
\item $\fn\pmod{\uf}\not\con 0$,
\item $\fn$ is a $\Hecke{B}{N^+}{p^n}$-eigenform, and $\lam_{\fn}\con \lam_f\pmod{\uf^n}$, where $\lam_{\fn}:\Hecke{B}{N^+}{p^n}_\cO=\Hecke{B}{N^+}{p^n}\ot_\Z\cO\surjto\cOn$ be the $\cO$-algebra homomorphism induced by $\fn$. Namely, we have the following equalities in $\cO_n$
\[\lam_{\fn}(T_\pme)=\al_\pme(f) \text{ for }\pme\ndivides pN^+\Delta,\quad \lam_{\fn}(U_\pme)=\al_\pme(f)\text{ for }\pme\divides pN,\quad \lam_{\fn}(\Dmd{a})=a^\frac{k-2}{2}\text{ for }a\in\Zp^\x.\]
\end{mylist}
\end{defn}
We fix an $n$-admissible form $\cD=(\Delta,\fn)$ and an $n$-admissible prime $\ell\ndivides \Delta$ with $\ep_\ell\cdot \al_\ell\con \ell+1\pmod{\uf^n}$, where $\ep_{\ell}$ is the sign as in (4) of \defref{D:admissibleprimes}. Let $B$ be the definite quaternion algebra over $\Q$ of discriminant $\Delta$. Write $\cU=\cU_{N^+,p^n}\subset \Bhat^\x$ for the open compact subgroup defined in \eqref{E:opncmpt}. Let $\bbT=\TBU_\cO$ and $\bbT^\Setl=\Hecke{B}{\ell N^+}{p^n}_\cO$. We extend $\lam_{\fn}$ to an $\cO$-algebra homomorphism $\lam_{\fn}^\Setl:\bbT^\Setl\to\cOn$ by defining $\lam_{\fn}^\Setl(U_\ell):=\ep_{\ell}$. Let $\cIfn$ (resp. $\cIll$) be the kernel of $\lam_{\fn}:\bbT\to\cOn$ (resp. $\lam_{\fn}^\Setl:\bbT^\Setl\to\cOn$). The eigenform $\fn$ gives rise to a surjective $\cO$-module map
\[\psi_{\fn}:\SB(\cU,\cO)/\cIfn\surjto \cOn,\,h\mapsto \psi_{\fn}(h):=\pair{\fn}{h}_\cU.\]



\begin{prop}\label{P:2}Assume that $\textup{($\mathrm{CR}^+$)}$ and \eqref{PO} hold. Then we have an isomorphism \[\psi_{\fn}:\SB(\cU,\cO)/\cIfn\isoto\cOn.\]
\end{prop}
\begin{proof} Let $P_k\subset \cIfn$ be the ideal of $\bbT$ generated by $\stt{\Dmd{a}-a^\frac{k-2}{2}\mid a\in\Zp^\x}$. Let $e=\lim_{n\to\infty} U_p^{n!}$ be Hida's ordinary projector on the space of \padic modular forms on $B$. Let $R_{N^+}$ and $R_{pN^+}$ be the Eichler orders of level $N^+$ and $pN^+$ in $B$. Let $\cV=\wh R_{N^+}^\x$ and $\cV_0(p)=\wh R^\x_{pN^+}$. By Hida theory for definite quaternion algebra (the case $q=0$ in \cite[Corollary\,8.2 and \,Proposition\,8.3]{Hida:p-adic-Hecke-algebra}), we have
\[e.\SB(\Un,\cOn)[P_k]=e.\SBk(\cV_0(p),\cOn).\]
Taking \pont dual, we find that
\[e.\SB(\Un,\cO)/(\uf^n,P_k)=e.\SBk(\cV_0(p),\cO)/(\uf^n).\]
Let $e^\circ=\lim_{n\to\infty} T_p^{n!}$ be the ordinary projector on $\SBk(\cV,\cO)$. Let $\frakm$ be the maximal ideal of $\bbT$ containing $\cI_{\fn}$.  Then the $p$-stabilization map gives rise to an isomorphism
\beq\label{E:pstablization}e^\circ.\SBk(\cV,\cO)_\frakm\isoto e.\SBk(\cV_0(p),\cO)_\frakm\eeq
under the assumption \eqref{PO}, and induces a surjective map $\bbT\to e^\circ.\bbT_B(\cV)_\cO$, which takes $U_p\mapsto u_p$, where $u_p$ is the unique unit root solution of $X^2-T_pX+p^{k-1}$ in $e^\circ.\bbT_B(\cV)$. Since $U_p-\al_p(f)\in \frakm$ with $\al_p(f)\in\cO^\x$, we find that \[\SB(\Un,\cO)_{\frakm}/(\uf^n,P_k)=\SBk(\cV_0(p),\cO)_{\frakm}/(\uf^n)\iso \SBk(\cV,\cO)_{\frakm}/(\uf^n).\]
By \cite[Proposition\,6.8]{Hsieh_Chida}, $\SBk(\cV,\cO)_{\frakm}$ is a cyclic $\TBU_{\frakm}$-module, and hence $\SB(\Un,\cO)/\cIfn$ is generated by some modular form $h$ as a $\TBU$-module. Since $\psi_{\fn}$ is surjective and Hecke operators in $\bbT$ are self-adjoint with respect to $\pairing_\cU$, it follows that $\psi_{\fn}(h)=\pair{\fn}{h}_\cU\in\cOn^\x$ and the annihilator of $h$ in $\bbT$ is $\cIfn$. Therefore, \[\SB(\cU,\cO)/\cIfn\iso \bbT/\cIfn=\cOn.\]
This completes the proof.
\end{proof}

 Let $B'$ be the indefinite quaternion algebra of discriminant $\Delta\ell$ and let $\ShCl=M_{\cU'}$ be the Shimura curve attached to $B'$ of level $\cU'$ introduced in \subsecref{SS:ShimuraCurve}. Let $\Jacl=J(\ShCl)$ be the Jacobian of $\ShCl$ and let $\Phi^\Setl$ be the group of connected components of the special fibre of the \Neron model of $\Jacl$ over $\cO_{\cK_\ell}$.

\begin{thm}\label{T:multiplicityOne}Let $\Phi^\Setl_\cO=\Phi^\Setl\ot_\Z\cO$. We have an isomorphism
\[\Phi^\Setl_\cO/\cIll\iso \SB(\cU,\cO)/\cIfn\stackrel{\psi_{\fn}}\isoto \cOn.\]
\end{thm}
\begin{proof}Let $\frakml$ be the maximal ideal of $\bbT^\Setl$ containing $\cIll$. Let $\cS=\SB(\cU,\cO)_{\frakml}$. The embedding $\cS\hookto \cS^{\oplus 2},\,x\mapsto (0,x)$ induces an isomorphism
\[\cS/(\ep_\ell T_\ell-\ell-1)\isoto\cS^{\oplus 2}/(\wtd U_\ell-\ep_\ell).\]
It is shown in \cite[Proposition 1.5.9 (1)]{Nekovar:Level_raising} that the quotient $\SB(\cU,\cO)^{\oplus 2}/(\SB(\cU,\cO)^{\oplus 2})_0$ is Eisenstein, so we find that $\cS^{\oplus 2}=((\SB(\cU,\cO)^{\oplus 2})_0)_{\frakml}$ (the ideal $\frakml$ is not Eisenstein). By \propref{P:1} we have a $\Hecke{B}{\ell N^+}{p^n}$-module isomorphism
\[( \Phi^\Setl_\cO)_{\frakml}\iso\cS^{\oplus 2}/(\wtd U_\ell^2-1)\iso\cS^{\oplus 2}/(\wtd U_\ell-\ep_\ell)\isoto\cS/(\ep_\ell T_\ell-\ell-1).\]
In particular, we see that $U_\ell$ acts on $\Phi^\Setl_\cO$ by $\ep_\ell$. Combined with \propref{P:2}, the theorem follows.
\end{proof}


Denote by $T_p(\Jacl)=\prolim_{m}\Jacl[p^m](\Qbar)$ the \padic Tate module of $\Jacl$.
\begin{cor}\label{C:levelraising}We have an isomorphism as $G_\Q$-modules
\[T_p(\Jacl)_\cO/\cIll\iso \Tfn.\]
\end{cor}
\begin{proof}Let $T^\Setl:=T_p(\Jacl)_\cO$. The argument in \cite[Theorem\,5.17 and the remarks below]{Bertolini_Darmon:IMC_anti}, based on the $\ell$-adic uniformization of $\Jacl(\Qbar_\ell)$ and the Eichler-Shimura congruence relation (\cf\cite[\S1.6.8]{Nekovar:Level_raising}), yields $T^\Setl/\frakml\iso T_{f,1}$ and the exact sequence
\beq\label{E:6}\Phi^\Setl_\cO/\cIll\to H^1(K_\ell,T^\Setl/\cIll)\to H^1_{\fin}(K_\ell,\cX^{[\ell]}_\cO/\cIll),\eeq
where $\cX^{[\ell]}_\cO=\cX^{[\ell]}\ot_\Z\cO$ and $\cX^{[\ell]}$ is the character group of $\Phi^{[\ell]}$. 
In addition, by the proof of \cite[Lemma 5.16]{Bertolini_Darmon:IMC_anti}, \thmref{T:multiplicityOne} implies that $T^\Setl/\cIll$ contains a cyclic $\cO$-submodule of order $\uf^n$. Thus, we find that $T^\Setl/\cIll\iso \cO/(\uf^n)e_1\oplus \cO/(\uf^r)e_2$ with $r\leq n$. Since the residual Galois representation $\ol{\rho}_f$ is absolutely irreducible,
we have an equality
\[\ol{\rho}_f(\bbF[G_\Q])=\End_{\bbF}(T_{f,1})=\End_{\cO}(T^\Setl/\frakml)\quad(\bbF=\cO/(\uf)).\]
In particular, there exists an element $h\in\rho_{f}(\cO[G_\Q])$ such that
\[he_2=ae_1+be_2,\,a\in\cO^\x,\,b\in\cO.\]
This implies that $\uf ^r e_1=0$, and hence $r=n$.
\end{proof}

Let $\wtd r_\ell:\Jacl(\cK_m)\to \Phi^\Setl_\cO\ot\cOn[\Gamma_m]$ be the reduction map
\[\wtd r_\ell(D)=\sum_{\sg\in \Gamma_m}r_\ell(\iota_\ell(\sg(D)))\sg.\]
\begin{thm}\label{T:commutative}
\begin{enumerate}
\item
There is an isomorphism
$$
\psi_{\fn}:\Phi^\Setl_\cO \slash \cIll \isoto H^1_{\sing}(K_\ell, \Tfn)
$$
which is canonical up to the choice of an identification of
$T_p(\Jacl) \slash \cIll$ with $\Tfn$.
\item
There is a commutative diagram
\[
\begin{CD}
\Jacl(\cK_m)\slash \cIll @>>> H^1(K_m, \Tfn)    \\
    @VV{\wtd r_\ell} V                   @VV{\partial_\ell} V  \\
\Phi^\Setl_\cO\slash \cIll\ot_\cO\cOn[\Gamma_m] @>\psi_{\fn}>\sim> H^1_{\sing}(K_{m,\ell}, \Tfn) ,
\end{CD}
\]
where the top horizontal map arises from the natural Kummer map,
and $\partial_\ell$ is the residue map.
Moreover, there is a similar commutative diagram
\[
\begin{CD}
\widehat{\Jacl}(K_\infty)\slash \cIll @>>> \widehat{H}^1(K_\infty, \Tfn)    \\
    @VV{\widetilde{r}_\ell} V                   @VV{\partial_\ell} V  \\
\Phi^\Setl_{\cO}\slash \cIll\ot_\cO\cOn\powerseries{\Gamma} @>\psi_{\fn}>\sim> \widehat{H}^1_{\sing}(K_{\infty,\ell}, \Tfn).
\end{CD}
\]
\end{enumerate}
\end{thm}
\begin{proof}
This is a direct generalization of \cite[Corollary 5.18]{Bertolini_Darmon:IMC_anti}.
See also \cite[Section 1.7.3]{Nekovar:Level_raising}.
In part (2), the lower horizontal map is deduced from the map $\psi_g$ in part (1) by using the identification $H^1_{\mathrm{sing}}(K_{m,\ell},T_{f,n})\cong H^1_{\mathrm{sing}}(K_{\ell},T_{f,n})\otimes_{\cO}\cO_n[\Gamma_m]$ as in Lemma \ref{L:2}.
\end{proof}

\subsection{Construction of the cohomology class $\kap_\cD(\ell)$}\label{SS:CohomologyClass}
\def\cmptv{\cmpt_\pme}
\def\w{w}
\def\setn{{(m)}}
In this section, we associate a cohomology class $\kappa_\cD(\ell)$ in the Selmer group $\wh\Sel_{\Delta\ell}(K_\infty,\Tfn)$ to an $n$-admissible form $\cD=(\Delta,\fn)$ and an $n$-admissible prime $\ell\ndivides \Delta$.
\subsubsection{}
 Fix a decomposition $N^+\cO_\cK=\frakN^+\ol{\frakN^+}$ once and for all. For each $\pme\divides N^+$, define $\cmptv\in \GL_2(\Q_\pme)$ by:
\beq\label{E:cmptv.W}\begin{aligned}
\cmptv=&\Diff^{-1}\pMX{\CMP}{\ol{\CMP}}{1}{1}\in\GL_2(\cK_\w)=\GL_2(\Qq)\text{ if $\pme=\w\wbar$ is split with $\w|\frakN^+$.}
\end{aligned}\eeq
For each positive integer $m$, we define $\cmpt_p^\setn\in \GL_2(\Qp)$ as follows.
If $p=\frakp\ol{\frakp}$ splits in $\cK$, we put
\begin{align}\label{E:op1.W}\cmpt_p^\setn=&\pMX{\CMP}{-1}{1}{0}\pDII{p^m}{1}\in\GL_2(\cK_\frakp)=\GL_2(\Qp).
\intertext{If $p$ is inert in $\cK$, then we put}
\label{E:op2.B}\cmpt_p^\setn=&\pMX{0}{1}{-1}{0}\pDII{p^m}{1}.\end{align}
We set
\beq\label{E:cmpt}\cmpt^\setn:=\cmpt_p^\setn\prod_{\pme\divides N^+}\cmptv\in \GL_2(\wh\Q_{(pN^+)})\iso\Bhat_{(pN^+)}^\x\hookto\Bhat^\x \eeq
Let $\cR_m=\Z+p^m\cO_\cK$ be the order of $\cK$ of conductor $p^m$.
It is not difficult to verify immediately that
\beq\label{E:optimal}(\cmpt^{(m)})^{-1}\wh\cR_m^\x\cmpt^{(m)}\subset \cU_{N^+,p^n}\text{ if }m\geq n. \eeq
We define a map
\[x_m:\Pic\cR_m=\cK^\x\bksl \wh\cK^\x/\wh\cR_m^\x\to \XB(\cU),\quad \cK^\x a\wh \cR_m^\x\mapsto x_m(a)=[a\cmpt^\setn]_\cU.\]
\subsubsection{}\label{SSS:Shimuracurve}
Let $\Un'=\vp_{B,B'}(\Un^\setl)\cO_{B_\ell}^\x\subset \wh B^{\prime\x}$ and let $\ShCl=M_{\Un'}$ be the Shimura curve of level $\Un'$. Let $m$ be a non-negative integer. To each $a\in \wh K^\x$, we associate the Heegner point $P_m(a)$ defined by
\beq\label{E:CMpoint}P_m(a):=[(z',\vp_{B,B'}(a^\setl\cmpt^\setn\tau_n))]_{\Un'}\in \ShCl(\C),\eeq
where $\tau_n\in\wh B_{(pN^+)}^\x$ is the Atkin-Lehner involution defined in \subsecref{SS:weighttwo}. Note that the level subgroup $\cU_{N^+,p^n}$ contains the subgroup $\wh\Z^\x$, so from \eqref{E:optimal} and Shimura's reciprocity law \eqref{E:ShimuraRec}, we deduce that
\[P_m(a)\in \ShCl(H_m)\cap \CM_\cK^{\ell-ur}(\ShCl)\text{ if }m\geq n\]
and that $P_m(b)^\sg=P_m(ab)$ for $\sg=\rec_\cK(a)\in G_m=\Gal(H_m/\cK)$.

Choose an auxiliary prime $\pme_0\ndivides p\ell N^+\Delta$ such that $1+\pme_0-\al_{\pme_0}\in \cO^\x$. We define
\begin{align*}\xi_{\pme_0}:\Div\ShCl(H_m)&\to \Jacl(H_m)_\cO=\Jacl(H_m)\ot_\Z\cO,\\
P&\mapsto \xi_{\pme_0}(P)=cl((1+\pme_0-T_{\pme_0})P)\ot (1+\pme_0-\al_{\pme_0})^{-1}.\end{align*}
Let $P_m:=P_m(1)$. Define
\begin{align*}
D_m=&\sum_{\sg\in \Gal(H_m/K_m)}\xi_{\pme_0}(P_m^\sg)\in \Jacl(K_m)_\cO.
\end{align*}
\def\Kummer{\mathrm{Kum}}
\def\Tfnchi{\Tfn}
Denote by $\Kummer: \Jacl(H_m)\ot_\Z\cO\to H^1(H_m, T_p(\Jacl)_\cO)$ the Kummer map. Define
\[\kappa_\cD(\ell)_m:=\al_p(f)^{-m}\cdot\Kummer(D_m)\pmod{\cIll}\in  H^1(K_m, T_p(\Jacl)_\cO/\cIll)=H^1(K_m,\Tfnchi).\]
Note that $\kappa_\cD(\ell)_m$ is independent of the choice of the auxiliary prime $\pme_0$.
The following lemma  says that the collection of classes $\stt{\kappa_\cD(\ell)_m}_m$ form a norm-compatible system.
\begin{lm}$\cores_{K_{m+1}/K_m}(D_{m+1})=U_p.D_m$.
\end{lm}
\begin{proof}This is a standard fact. For example, see Longo-Vigni \cite[Proposition 4.8]{LongoVigni:HidaFamilies}
(their setting is slightly different, but the proof is identical).
It is basically a consequence of Shimura's reciprocity law.
\end{proof}
Finally, define the cohomology class $\kap_\cD(\ell)$ associated to an $n$-admissible form $\cD=(\Delta,\fn)$ and an $n$-admissible prime $\ell$ by
\[\kappa_\cD(\ell)=(\kappa_\cD(\ell)_m)_m\in \wh H^1(K_\infty,\Tfnchi).\]

\begin{prop}The cohomology class $\kappa_\cD(\ell)$ belongs to $\wh\Sel_{\Delta\ell}(K_{\infty},\Tfnchi)$.
\end{prop}
\begin{proof}This should be well-known to experts. We sketch a proof here for the convenience of the readers. We need to show that for each integer $m\geq n$,
\begin{mylist}
\item $\partial_q(\kappa_\cD(\ell)_m)=0$ for $q\ndivides p\Delta\ell$,
\item $\res_q(\kappa_\cD(\ell)_m)\in H_\Ord^1(K_{m,q},\Tfn)$ for $q\divides p\ell\Delta$.
\end{mylist}
Part (1) follows from the fact that $\Jacl$ has good reduction at primes $q\ndivides p\Delta\ell N^+$ and \lmref{L:2} (1) for $\pme\divides N^+$. If $q\divides \ell\Delta$, then part (2) is a standard consequence of the description of the $q$-adic uniformization of $\Jacl$ at toric reduction primes $q\divides\Delta\ell$. It remains to show part (2) for $q=p$.
Let $\cI=\cIll$ and let $\frakm^{[\ell]}$ be the maximal ideal of $\bbT^{[\ell]}$ containing $\cI$. Let $T=T_p(\Jacl)_\cO$. Then the localization $T_{\frakm^{[\ell]}}$ at $\frakm^{[\ell]}$ is a direct summand of $T$, and we have maps as $\bbT^{[\ell]}[G_\Q]$-modules
\[ T\to T_{\frakm^{[\ell]}}\to T_{\frakm^{[\ell]}}/\cI=T/\cI\iso \Tfn.\]
Let $E$ be the fractional field of $\cO$. The $G_\Q$-module $V_{\frakm^{[\ell]}}:=T_{\frakm^{[\ell]}}\ot_\cO E$ is a direct sum of \padic Galois representations $\rho_{g}\ot\ep$ attached to $p$-ordinary elliptic new forms $g$ of weight two and nebentype $\ep^{-2}$ with $\rho_{g}\ot\ep\con \rho_f(\frac{2-k}{2})\pmod{\frakm_\cO}$. In addition, there is an exact sequence as $\bbT^{[\ell]}[G_{\Qp}]$-modules
\[\exact{F_p^+V_{\frakm^{[\ell]}}}{V_{\frakm^{[\ell]}}}{V_{\frakm^{[\ell]}}/F_p^+V_{\frakm^{[\ell]}}}\]
such that the inertia group of $G_{\Qp}$ acts on $F_p^+V_{\frakm^{[\ell]}}$ (resp. $V_{\frakm^{[\ell]}}/F_p^+V_{\frakm^{[\ell]}})$ via $\e_\bbT\e$ (resp. $\e_\bbT^{-1}$), where $\e_\bbT:G_{\Qp}\to(\bbT^{[\ell]})^\x,\,\sg\mapsto \Dmd{\e(\sg)}$. Let $F_p^+T_{\frakm^{[\ell]}}:=F_p^+V_{\frakm^{[\ell]}}\cap T_{\frakm^{[\ell]}}$. Then it is not difficult to see that $F_p^+T_{\frakm^{[\ell]}}/\cI\iso F_p^+\Tfn$ as $G_{\Qp}$-modules. Consider the commutative diagram
\[\begin{CD} H^1(K_{m,p},T_{\frakm^{[\ell]}})@>\al>> H^1(K_{m,p},T_{\frakm^{[\ell]}}/F_p^+T_{\frakm^{[\ell]}})\\
@V\beta VV @V\beta^-VV\\
H^1(K_{m,p},V_{\frakm^{[\ell]}})@>\al_E >> H^1(K_{m,p},V_{\frakm^{[\ell]}}/F_p^+V_{\frakm^{[\ell]}}).
\end{CD}\]
Let $\kappa(D_m)_\frakm$ be the image of $\kappa(D_m)$ in $H^1(K_{m,p},T_{\frakm^{[\ell]}})$. To prove the proposition, it suffices to show that $\al(\kappa(D_m)_\frakm)=0$.

 By \cite[Example\,3.11]{BlochKato}, $\beta(\kappa(D_m)_\frakm)$ belongs to the local Bloch-Kato Selmer group $H^1_f(K_{m,p},V_{\frakm^{[\ell]}})$, which in turn implies that $\kappa(D_m)_\frakm$ lies in the kernel of the composition $\al_E\circ \beta=\beta^-\circ\al$ in view of \cite[Proposition\,12.5.8]{Nekovar:SelmerComplexes}. On the other hand, the map $\beta^-$ is injective by \eqref{PO} (since $H^0(K_{m,p},T_{\frakm^{[\ell]}}/F_p^+T_{\frakm^{[\ell]}}\ot_\cO E/\cO)=0$ in view of \lmref{L:5}), so we conclude that $\al(\kappa(D_m)_\frakm)=0$.
\end{proof}

\section{First and second explicit reciprocity laws}\label{S:Reciprocity}
\subsection{First explicit reciprocity law}Let $\cD=(\Delta,\fn)$ be an $n$-admissible form. Define
\[\Theta_m(\cD)=\al_p(f)^{-m}\sum_{[a]_m\in G_m}\fn(x_m(a))[a]_m\in \cOn[G_m].\]
Here $[a]_m:=\rec_\cK(a)|_{H_m}\in G_m$ is the map induced by the geometrically normalized reciprocity law. Then $\Theta_{m-1}(\cD)$ coincides with the image of $\Theta_{m}(\cD)$ under the natural quotient map $G_m\to G_{m-1}$, since $\fn$ is an $U_p$-eigenform with the eigenvalue $\al_p(f)$. Let $\pi_m:G_m\to \Gamma_m=\Gal(\cK_m/\cK)$ be the natural quotient map and let \beq\label{E:theta_elt}\theta_m(\cD)=\pi_m(\Theta_m(\cD))\in\cOn[\Gamma_m],\quad \theta_\infty(\cD)=(\theta_m(\cD))_m\in\cOn\powerseries{\Gamma}.\eeq
\begin{thm}[First explicit reciprocity law]\label{T:first}For $m\geq n>0$, we have 
\[\partial_\ell(\kappa_\cD(\ell)_m)=\theta_m(\cD)\in \cOn[\Gamma_m].\]
Therefore,
\[\partial_\ell(\kappa_\cD(\ell))=\theta_\infty(\cD)\in\cOn\powerseries{\Gamma}.\]
\end{thm}
\begin{proof}It follows from the commutative diagram \eqref{E:comutativediagram} and \eqref{E:specialization} that
\[\begin{aligned}\psi_{\fn}(r_\ell(D_m^{\sg}))=&\sum_{[b]_m\in\Gal(H_m/\cK_m)}\pair{\fn}{x_m(ab)\tau_n}_\cU=\sum_{[b]_m\in\Gal(H_m/\cK_m)}\fn(x_m(ab))\\
&\quad(\sg=\pi_m([a]_m)\in \Gal(\cK_m/\cK),\,a\in\wh \cK^\x).\end{aligned}\]
Therefore, by \thmref{T:commutative}
\begin{align*}\partial_\ell(\kap_\cD(\ell)_m)=&\sum_{\sg\in \Gamma_m}\psi_{\fn}(r_\ell(D_m^{\sg}))\sg\\
=&\sum_{[a]_m\in G_m}\fn(x_m(a))\pi_m([a]_m)=\theta_m(\cD).\qedhere
\end{align*}
\end{proof}
\begin{remark}This equality depends on the choices of the embedding $\iota_\ell:\Qbar\hookto \C_\ell$ and the isomorphism $T_p(\Jacl)/\cIll\iso \Tfn$. Different choices result in a unit factor in $\cO\powerseries{\Gamma}$.
\end{remark}
\subsection{Ihara's lemma}In this subsection, we retain the notation in \secref{S:ShimuraCurve}. Ihara's lemma is the key ingredient in proof of the second explicit reciprocity law in \cite{Bertolini_Darmon:IMC_anti}.
We recall the following version of Ihara's lemma due to Diamond and Taylor \cite[Theorem\,2]{Diaomond_Taylor:Inventione}. Let $B'$ be an indefinite quaternion algebra of discriminant $\Delta_{B'}$. Let $\cV$ be an open compact subgroup $\cV$ of $\wh B^{\prime\x}$. Let $M_{\cV}$ be the associated Shimura curve. For each \padic ring $A$, let $\sF_k(A)$ be the local system on $M_{\cV}{}_{/\Cp}$ attached to $L_{k}(A)$ and let $\sL_k(\cV,A)=H_{\et}^1(M_{\cV}{}_{/\Cp},\sF_k(A))$. Then $\sL_2(\cV,\bbF)=H^1_\et(M_{\cV}{}_{/\Cp},\bbF)$. If $q$ is a prime, let $\cV_0(q)=\cV\cap R'_{q}$, where $R'_q$ is an Eichler order of level $\ell$ in $B'$. Let $\ell\ndivides p\Delta_{B'}$ be a rational prime. Denote by $\bbF=\cO/(\uf)$ the residue field of $\cO$. Let $\frakm$ be the maximal ideal of $\bbT_{B'}(N^+,p^n)\iso\bbT_{B}(\ell N^+,p^n)$ containing the kernel of the ring homomorphism $\lam_\fn^\Setl$ defined below \defref{D:1}. Then $\frakm$ is an ordinary and non-Eisenstein maximal ideal.
\begin{thm}[Ihara's lemma]\label{T:DiamondTaylor}If $\cV$ is maximal at the prime $p$ and $\ell$, then we have an injective map
\[1+\eta_{\ell}:\sL_k(\cV,\bbF)_\frakm^{\oplus 2}\hookto \sL_k(\cV_0(\ell),\bbF)_\frakm,\]
where $\eta_\ell$ is the degeneracy map at $\ell$.
\end{thm}

We will need Ihara's lemma for the open compact subgroup $\cU'$. However, $\cU'=\cU_{N^+,p^n}'$ is not maximal at $p$, so we can not apply \thmref{T:DiamondTaylor} directly.

\begin{lm}\label{L:1}Suppose that \eqref{PO} holds. Then
\[\sL_k(\cV,\bbF)_\frakm\iso \sL_k(\cV_0(p),\bbF)_\frakm.\]
\end{lm}
\begin{proof} We first note that the assumption \eqref{PO} implies that the injective map
\beq\label{E:5}\sL_k(\cV,\Cp)_\frakm\to \sL_k(\cV_0(p),\Cp)_\frakm\eeq
is an isomorphism. Indeed, by the Eichler-Shimura isomorphism, the cokernel $C$ of the map \eqref{E:5} is two copies of the space of ordinary modular forms on $B'$ which are new at $p$ and hence $U_p^2-p^{k-2}$ annihilates $C$.
We thus conclude that the cokernel of the injective map
\[i:\sL_k(\cV,\cO)_\frakm\hookto \sL_k(\cV_0(p),\cO)_\frakm\]
is torsion. By \cite[Lemma 4]{Diaomond_Taylor:Inventione}, for a sufficiently small open compact subgroup $\cU$, $\sL_k(\cU,\cO)_\frakm$ is torsion-free and $\sL_k(\cU,\cO)_\frakm\ot\bbF=\sL_k(\cU,\bbF)_\frakm$. This implies that $i$ is an isomorphism, and so is $i\ot\bbF$.
\end{proof}
\begin{cor}\label{C:1}Let $\cU'=\cU_{N^+,p^n}'$. Suppose that \eqref{PO} holds. Then
\[1+\eta_{\ell }:\sL_2(\cU',\bbF)_\frakm[P_k]^{\oplus 2}\to \sL_2(\cU'_0(\ell),\bbF)_\frakm\]
is injective.
\end{cor}
\begin{proof}Let $R'_{N^+}$ be an Eichler order of level $N^+$ in $B'$. Let $\cV=\wh R_{N^+}^{\prime \x}$. By Hida theory for indefinite quaternion algebras (the case $q=1$ in \cite[Corollary\,8.2 and Proposition\,8.3]{Hida:p-adic-Hecke-algebra}), we have
\[ \sL_2(\cU',\bbF)_\frakm[P_k]\iso \sL_k(\cV_0(p),\bbF)_\frakm,\quad \sL_2(\cU'_0(\ell),\bbF)_\frakm[P_k]\iso \sL_k(\cV_0(\ell p),\bbF)_\frakm.\]
Combined with \thmref{T:DiamondTaylor} and \lmref{L:1}, the corollary follows immediately.
\end{proof}
\subsection{Second explicit reciprocity law}
\def\gn{g''}
Let $\ell_1,\ell_2\ndivides \Delta$ be two $n$-admissible primes. Let $B''$ be the definite quaternion algebra of discriminant $\Delta_{B''}=\Delta\ell_1\ell_2$. Let $\ell=\ell_2$. We briefly discuss the reduction of CM points $P_m\in M_n^{[\ell_1]}(H_m)$ modulo $\ell$, where $M_n^{[\ell_1]}$ is the Shimura curve defined by the analogues manner in \secref{SSS:Shimuracurve}. Let $\cM_n^{[\ell_1]}$ be the Kottwitz model of $M_n^{[\ell_1]}\ot_{\Q}\cK_{\ell}$ over $\Z_{\ell^2}=\cO_{\cK_{\ell}}$. We recall that, for a $\Z_{\ell^2}$-algebra $R$ and a geometric point $\ol{s}$ in $\Spec R$,
$\cM_n^{[\ell_1]}(R)$ consists of prime-to-$\ell$ isogeny classes of triples $[(A,\iota,\eta\cU')]$, where
\begin{itemize}
\item $A$ is an abelian surface over $R$,
\item $\iota:\cO_{B'}\hookto \End A\ot \Z_{(\ell)}$ satisfies the Kottwitz determinant condition,
\item $\eta: \wh B^{\prime(\ell)}\iso V^{(\ell)}(A_{\ol{s}})=T^{(\ell)}(A_{\ol{s}})\ot \Q$ is an isomorphism of $\cO_{B'}$-modules in the sense that
\[\eta(bx)=\iota(b)\eta(x)\text{ for all }b\in\cO_{B'},\]
and $\eta\cU'=\stt{\eta\cdot u\mid u\in \cU'}$ is a $\cU'$-equivalence class of isomorphisms. Here $\eta\cdot u(x):=\eta(x\ol{u})$.
\end{itemize}
Now let $E_{/H_0}$ be an elliptic curve with CM by $\cO_\cK$ defined over the Hilbert class field $H_0$ of $\cK$ and let $\cE$ be the \Neron model of $E\ot_{H_0,\iota_{\ell}}\cK_{\ell}$ over $\cO_{\cK_{\ell}}$. Let $\cA=\cE\x_{\cO_{\cK_\ell}} \cE$ and $\cA_s=\cA\ot\bbF_{\ell^2}$. Then $B''=\End_{B'}^0(\cA_s)$ is the commutant of $B'$ in $\End^0(\cA_s)=M_2(D)$, where $D$ is the definite quaternion algebra ramified at $\ell$ and $\infty$. With a suitable endomorphism and level structure $(\iota_\cA,\eta_\cA)$ on $\cA$, the CM point $P_0=[z',1]_{\cU'}$ is represented by $[(\cA,\iota_\cA,\eta_\cA\cU')]\in M_n^{[\ell_1]}(\cK_{\ell})=\cM_n^{[\ell_1]}(\cO_{\cK_{\ell}})$. Let $\vp_{B'',B'}:\wh B^{\prime\prime(\ell)\x}\iso \wh B^{\prime(\ell)\x}$ be the unique isomorphism such that the following diagram commutes:
\[\begin{CD}\wh B^{\prime(\ell)\x}@>\sim>\eta> V^{(\ell)}(\cA_{\ol{s}}) \\
@V\text{ right multiplication by }\vp_{B'',B'}(b'')VV @Vb''VV\\
\wh B^{\prime(\ell)\x}@>\sim>\eta> V^{(\ell)}(\cA_{\ol{s}}).
\end{CD}
\]
Let $\cU''=\vp_{B'',B'}^{-1}(\cU^{\prime(\ell)})\cO_{B''_{\ell}}^\x\subset \wh B^{\prime\prime\x}$. Henceforth, we will identify $\wh B^{(\ell_1\ell)\x}\iso \wh B^{\prime\prime(\ell_1\ell)\x}$ via $\vp_{B'',B'}^{-1}\circ\vp_{B,B'}$.


\begin{thm}[Second explicit reciprocity law]\label{T:second}
Assume that $\textup{($\mathrm{CR}^+$)}$ holds. Then there exists an $n$-admissible form $\cD''=(\Delta\ell_1\ell_2,\gn)$ such that
\[v_{\ell_1}(\kappa_\cD(\ell_2))=v_{\ell_2}(\kappa_\cD(\ell_1))=\theta_\infty(\cD'')\in \cOn\powerseries{\Gamma}.\]
\end{thm}
\begin{proof}
Let \[\gamma:X_{B''}(\cU'')=B^{\prime\prime\x}\bksl \wh B^{\prime\prime\x}/\cU^{\prime\prime}=
(B^{\prime\prime\x}\cap \cO_{B''_{\ell_2}}^\x)\bksl \wh B^{\prime\prime(\ell_2)\x}/\cU^{\prime\prime(\ell_2)}\longto \cM_n^{[\ell_1]}(\bbF_{\ell_2^2})\] be the map identifying $X_{B''}(\cU'')$ with the set of supersingular points in $\cM_n^{[\ell_1]}(\bbF_{\ell_2^2})$ defined by
\begin{align*}\gamma([b'']_{\cU''})=&[(\cA_s,\iota_s,\eta_s\cdot\vp_{B'',B'}(b'')\cU'))]\\
&(b''\in \wh B^{\prime\prime(\ell_2)\x},\,(\cA_s,\iota_s,\eta_s)=(\cA,\iota_\cA,\eta_\cA)\ot\bbF_{\ell_2^2}).\end{align*}
By the definition \eqref{E:CMpoint} of Heegner points, we have
\[P_m(a)=[(\cA,\iota_\cA,\eta_\cA\cdot\vp_{B,B'}(x_m(a)\tau_n)\cU')],\]
and hence
\[P_m\pmod{\ell}=\gamma(x_m(a)\tau_n).\]
Let $\cJ^{[\ell_1]}_n=\Pic^0_{\cM_n^{[\ell_1]}/\Z_\ell}$. Then $\gamma$ in turn induces a $\bbT_{B''}^{(\ell_2)}(N^+,p^n)$-module (via Picard functoriality) map
\[\gamma_*:\Z[X_{B''}(\cU'')]\to \cJ^{[\ell_1]}_n(\F_{\ell_2^2})_\cO/\cI_{\fn}^{[\ell_1]},\quad x\mapsto \xi_{\pme_0}(\gamma(x))=cl((T_{\pme_0}-\pme_0-1)\gamma(x))\ot (\al_{\pme_0}-1-\pme)^{-1}. \]
By a result of Ihara \cite[Remark G,\,page 19]{Ihara} and Ihara's lemma (\corref{C:1}) implies that $\gamma_*$ is indeed surjective. Therefore, we obtain a surjective map
\[\gamma_*:\Z[X_{B''}(\cU'')]\surjto \cJ^{[\ell_1]}_n(\F_{\ell_2^2})_\cO/\cI_{\fn}^{[\ell_1]}\surjto H^1_{\fin}(\cK_{\ell_2},\Tfn)\iso \cOn,\]
Therefore, $\gamma_*$ gives rise to a unique modular form $\gn\in\cS^{B''}(\cU'',\cOn)$ such that $\gamma_*(h)=\pair{h}{\gn}_{\cU''}$ for all $h\in \Z[X_{B''}(\cU'')]$ via the identification \eqref{E:weighttwo}. By definition, $\gn$ is an eigenform of $\bbT_{B''}^{(\ell_1)}(N^+,p^n)$. Moreover, by \cite[Lemma 9.1]{Bertolini_Darmon:IMC_anti}, $\gn$ is an eigenform of the $U_{\ell_2}$-operator with eigenvalue $\ep_{\ell_2}$. Define an $\cO$-algebra homomorphism $\lam_f^{[\ell_1\ell_2]}:\bbT_{B''}(N^+,p^n)_\cO\to \cOn$ by $t\mapsto\lam_f(t)$ if $t\in \bbT_{B''}^{(\ell_1\ell_2)}(N^+,p^n)$, $U_{\ell_1}\mapsto \ep_{\ell_1}$ and $U_{\ell_2}\mapsto \ep_{\ell_2}$. Then we conclude that $\gn$ is an eigenform of $\bbT_{B''}(N^+,p^n)$ such that $t\cdot \gn =\lam_f^{[\ell_1\ell_2]}(t) \gn$. Such an eigenform is unique up to $\cO^\x$, since $\cS^{B''}(\cU'',\cO)_{\frakm''}$ is a cyclic $\bbT_{B''}(N^+,p^n)_\cO$-module by \cite[Proposition\,6.8]{Hsieh_Chida}. By definition, we verify that \begin{align*}v_{\ell_2}(\kappa_\cD(\ell_1)_m)=&\sum_{[a]_m\in G_m}\gamma_*(x_m(a)\tau_n)\pi([a]_m)\\
=&\sum_{[a]_m\in G_m}\pair{x_m(a)\tau_n}{\gn}_{\cU''}\pi([a]_m)\\
=&\sum_{[a]_m\in G_m}\gn(x_m(a))\pi([a]_m).\end{align*}
 The rest of the assertions follow from the discussion in \cite[pp.61--62]{Bertolini_Darmon:IMC_anti}.
\end{proof}


\section{Euler system argument}\label{S:Euler}
\subsection{Preliminaries for Euler system argument}
In this subsection, we assume \eqref{Irr}, \eqref{PO} and \eqref{CR${}^+$3}. Let $H=\bar\rho_f^*(G_\Q)\subset \GL_2(\Fpbar)$ be the image of the residual Galois representation $\bar\rho_f^*=\rho_{f}(\frac{2-k}{2})\pmod{\uf}$.
\begin{lm}\label{L:6}The group $H$ contains the scalar matrix $-1$.
\end{lm}
\begin{proof}Let $Z$ be the center of $H$. By \eqref{Irr}, $Z$ must be contained in the group  $\Fpbar^\x$ of scalar matrices. If $p\divides \#H$, then it is well-known that $H$ contains a conjugate of $\SL_2(\Fp)$ (\cf\cite[Corollary\,2.3]{Ribet:Pacific}), and hence $-1\in H$. Assume $H$ has prime-to-$p$ order, then $H/Z$ must be cyclic, dihedral, or the three exceptional groups $S_4,A_4,A_5$. By \eqref{Irr}, $H/Z$ cannot be cyclic. Let $I_p$ be the image of an inertia group at the prime $p$ in $H/Z$, which is a cyclic group of order greater than $5$ by the assumption $\#(\Fp^\x)^{k-1}>5$ in \eqref{CR${}^+$1}. This excludes the possibility $H/Z$ being the other three exceptional groups. Therefore, $H/Z$ is dihedral and $I_p$ is contained in the cyclic group of even order ($k-1$ is odd). This implies that the center of $H/Z$ is of order two or four, from which it is easy to deduce that $Z$ contains an element of order two.
\end{proof}
\begin{lm}\label{L:CT} There exists an element $h\in H$ such that $\Tr(h)=\det(h)+1$ with $\det(h)\not =\pm 1\in \Fp$.
\end{lm}
\begin{proof}If $p$ does not divide $\ell^2-1$ for some $\ell\divides N^-$, then we can take $h$ to be the Frobenius $\Frob_\ell$. Otherwise, $\bar\rho_f$ is ramified at least one prime $\ell\divides N^-$ by \eqref{CR${}^+$3}, so $H$ contains the image of an inertia group at $\ell$ whose order is divisible by $p$. We will show by arguments in \cite{Ribet:Pacific} that $H$ contains a conjugate of $\GL_2(\Fp)$, and the lemma follows immediately. Since $p\divides \# H$ and $H$ is irreducible, we may assume $H/Z\subset \PGL_2(F) \subset \PGL_2(\Fpbar)$ and $H$ contains the group $\PSL_2(F)$ for some finite extension $F/\Fp$ by a well-known result of Dickson. In other words, $H\subset \Fpbar^\x\GL_2(F)$ and $\SL_2(F)\subset H\Fpbar^\x$. In particular, $H$ contains the commutator $[H,H]\supset [\SL_2(F),\SL_2(F)]=\SL_2(F)$. Let $C:=H\cap \GL_2(F)$ be a normal subgroup of $H$. Then $\SL_2(F)\subset C\subset \GL_2(F).$
Let $\gamma\in H$ be a generator of the image of an inertia group $I_p$ at the prime $p$. Then $\det \gamma=u$ with $u$ a generator of $\Fp^\x$, and the trace $\Tr(\sg)=u^\frac{k}{2}+u^{1-\frac{k}{2}}\in\Fp^\x$ is non-zero as $\#(\Fp^\x)^{k-1}>2$. On the other hand, write $\gamma=\gamma_0\lam$ with $\gamma_0\in\GL_2(F)$ and $\lam\in\Fpbar^\x$. Then $\lam=\Tr(\gamma_0)/\Tr(\gamma)\in F^\x$. We find that $\gamma\in C$ and $\det C=\Fp^\x$. From this analysis, we conclude that
\[C=\stt{h\in\GL_2(F)\mid \det h\in \Fp^\x}\subset H\]
and $H$ contains $\GL_2(\Fp)$. This completes the proof.
\end{proof}

\begin{thm}\label{T:Theorem3.2}
Let $s\leq n$ be positive integers. Let $\kappa$ be a non-zero element in $H^1(K,A_{f,s})$.
Then there exist infinitely many $n$-admissible primes $\ell$ such that $\partial_\ell(\kappa)=0$ and
the map
$$
v_\ell:\langle \kappa \rangle \to H^1_\fin(K_\ell, A_{f,s})
$$
is injective, where $\langle \kappa \rangle$ is the $\cO$-submodule of $H^1(K,A_{f,s})$
generated by $\kappa$.
\end{thm}
\begin{proof}
We follow the proof of \cite[\th 3.2]{Bertolini_Darmon:IMC_anti} with some modification. Suppose that $\langle\kappa\rangle\iso\cO/\uf^r\cO$. Since the map $H_\fin^1(K_\ell,A_{f,1})\to H_\fin^1(K_\ell,A_{f,s})$ induced by the inclusion is injective for all $n$-admissible primes $\ell$, replacing $\kappa$ by $\uf^{r-1}\kappa$, we may assume $s=1$ and $\kappa\in H^1(K,A_{f,1})$. Let $F_n=\overline{\Q}^{\mathrm{Ker} \rho_n}$ be the finite extension cut out by $\rho_n:=\rho_f^*\pmod{\uf^n}: G_{\Q}\to \mathrm{Aut}_{\cO}(\Afn)=\GL_2(\cOn)$. Since $\rho_n$ is unramified outside $Np$ and $(D_K,Np)=1$,
$K$ and $F_n$ are linearly disjoint.
Put $M=KF$.
Let $\tau$ be the non-trivial element in $\Gal(K\slash \Q)$.
Then $\Gal(M\slash \Q)=\Gal(K\slash \Q)\times \Gal(F_n\slash \Q)$
can be identified with the subgroup of $\langle \tau \rangle \times \mathrm{Aut}_{\cO}(\Afn)$.
Therefore we may write an element of $\Gal(M\slash \Q)$ as a pair $(\tau^j, \sg)$
with $j\in \{ 0,1 \}$ and $\sg\in \mathrm{Aut}_{\cO}(\Afn)$. By \lmref{L:6} the image $H$ contains the scalar matrix $-1$, which in turn implies that $\rho_n(\Gal(F_n/\Q))$ also contains the scalar matrix $-1$. It follows that $H^1(M/K,\Afn)=H^1(F/\Q,\Afn)=0$ (\cite[\lemma 2.13]{BD:Duke:DerivedMT}).
Therefore, the restriction map $H^1(K,A_{f,1})\to H^1(M,A_{f,1})=\Hom(G_M,A_{f,1})$ is injective. Let
$M_\kappa$ be the (non-trivial) extension on $M$ cut out by the image $\overline{\kappa}$ of $\kappa$
under the restriction to $H^1(M,A_{f,1})=\mathrm{Hom}(G_M,A_{f,1})$. Let $C_\kappa:=\bar\kappa(\Gal(M_\kappa/M))\subset A_{f,1}$ be a $\Fp[\Gal(M/\Q)]$-submodule, and $\dim_{\Fp}C_\kappa=2$ by \eqref{Irr}.

Assume without loss of generality that $\kappa$ belongs to an eigenspace for the action of $\tau$ so that $\tau \kappa = \delta \kappa$ for some $\delta \in \{ \pm 1 \}$.
Under this assumption, the extension $M_\kappa\slash \Q$ is Galois. Moreover, $\Gal(M_\kappa \slash \Q)$ is identified with the group $A_{\kappa}\rtimes \Gal(M\slash \Q)$,
where $\Gal(M\slash \Q)$ acts on $A_{\kappa}$ by the rule
$(\tau^j,\sg )(v)=\delta^j\rho_1(\sg) v$ for $(\tau^j,\sg) \in \Gal(M\slash \Q)$
and $v \in A_{\kappa}$.
As $-1\in \rho_n(\Gal(F_n/\Q)$, it follows from \lmref{L:CT} that we can choose a triple
$(v,\tau,\sg)$ as an element of $\Gal(M_\kappa\slash \Q)$ such that
\begin{enumerate}
\renewcommand{\labelenumi}{(\roman{enumi})}
\item $\rho_n(\sg) \in \GL_2(\cOn)$ has eigenvalues $\delta(=\pm 1)$ and $\lambda$,
where $\lambda \in (\Z/p^n\Z)^\times$ is not equal to $\pm 1\pmod{p^n}$ and the order of $\lambda$ is prime to $p$,
\item the element $v\not =0\in A_{\kappa}$ and belongs to
the $\delta$-eigenspace for $\sg$.
\end{enumerate}
By Chebotarev density theorem, there exist infinitely many primes $\ell$ with $\ell \nmid N$
such that $\ell$ is unramified in $M_{\kappa}\slash \Q$ and satisfies
$$\mathrm{Frob}_{\ell}(M_\kappa\slash \Q)=(v,\tau,\sg).$$
Then $\mathrm{Frob}_{\ell}(M\slash \Q)=(\tau,\sg)$ implies that $\ell$ is $n$-admissible.
For each prime $\mathfrak{l}$ of $M$ above $\ell$, let $d$ be the (even) degree of the
residue field corresponding to $\mathfrak{l}$. Then we have
$$
\mathrm{Frob}_{\mathfrak{l}}(M_\kappa\slash M)=(v,\tau,\sg)^d=v+\delta \sg v+\sg^2v+\cdots +\delta \sg^{d-1}=dv.
$$
Since $d$ is prime to $p$, $\overline{\kappa}(\mathrm{Frob}_{\mathfrak{l}}(M_\kappa\slash M))=d \overline{\kappa}(v)\not =0$ and hence $v_\ell(\kappa)=(\overline{\kappa}(\mathrm{Frob}_{\mathfrak{l}}(M_\kappa\slash M)))_{\frakl|\ell}\not =0$. This finishes the proof.
\end{proof}

Let $\Delta$ be a square-free integer such that $\Delta\slash N^-$ is a product of $n$-admissible primes.
\begin{defn}[$n$-admissible set]
A finite set $S$ of primes is said to be $n$-admissible for $f$ if
\begin{enumerate}
\item All $\ell \in S$ are $n$-admissible for $f$,
\item The map $\Sel_\Delta(K,\Tfn)\to \bigoplus_{\ell \in S}H_\fin^1(K_{\ell},\Tfn)$ is injective.
\end{enumerate}
\end{defn}
\begin{prop}\label{P:enlarge}
Any finite collection of $n$-admissible primes can be enlarged to an $n$-admissible set.
\end{prop}
\begin{proof}
It is a simple application of \thmref{T:Theorem3.2}.
\end{proof}
\subsection{Control theorem (II)}
Let $S$ be an $n$-admissible set for $f$. In this subsection, we prove control theorems for the compact Selmer group $\wh\Sel_\Delta^S(K_\infty,\Tfn)$. We begin with some preparations. Let $L/K$ be a finite extension in $K_\infty$.
\begin{lm}\label{L:injective}The natural map
$$
\Sel_\Delta(L,\Tfn)\to \bigoplus_{\ell\in S}H^1_\fin(L_{\ell},\Tfn)
$$
is injective.
\end{lm}
\begin{proof}
Let $C$ be the kernel of the map $\Sel_\Delta(L,\Tfn)\to \bigoplus_{\ell\in S}H^1_\fin(L_{\ell},\Tfn)$.
If $C$ is non-zero, then there exist a non-trivial element $\kappa$ in $C$ fixed by $\Gal(L/K)$ as $\Gal(L/K)$ is a $p$-group. Thus, $\kappa$ belongs to $\Sel_\Delta(K,\Afn)$ by \propref{P:control1}, and the image of $\kappa$
under the map  $\Sel_\Delta(K,\Tfn)\to \bigoplus_{\ell\in S}H^1_\fin(K_\ell,\Tfn)$ is zero.
This contradicts with the definition of $n$-admissible sets.
\end{proof}

\begin{lm}\label{L:exactness}We have an exact sequence
$$
0\to \Sel_\Delta(L,\Tfn)\to \Sel_\Delta^S(L,\Tfn) \to \bigoplus_{\ell\in S}H_\sing^1(L_{\ell},\Tfn)\to \Sel_\Delta(L,\Tfn)^\vee\to 0.
$$
\end{lm}
\begin{proof}We have seen in \propref{P:annihilator} and \propref{P:annihilator_ord} that the local conditions defining the Selmer group $\Sel_\Delta(L,\Tfn)$ and $\Sel_\Delta(L,\Afn)$ are orthogonal complements of each other. Therefore, by Poitou-Tate duality (\cf\cite[\th 1.7.3]{Rubin:EulerSystem}) we have an exact sequence
$$
0\to \Sel_\Delta(L,\Tfn)\to \Sel_\Delta^S(L,\Tfn)\to \bigoplus_{\ell\in S}H_\sing^1(L_{\ell},\Tfn)\to\Sel_\Delta(L,\Afn)^\vee.
$$
The last map is indeed surjective by \lmref{L:injective}.
\end{proof}

\begin{prop}\label{P:step1}The $\cOn[\Gal(L/K)]$-module $\Sel_\Delta^S(L,\Tfn)$ is free of rank $\# S$.
\end{prop}
\begin{proof}This is a direct generalization of the proof of \cite[\th 3.2]{BD:Duke:DerivedMT} after we replace \lemma 2.19 and \lemma 3.1 \loccit with \propref{P:control1} and \lmref{L:exactness} respectively.
\end{proof}

\begin{remark}
If $f$ is a new form attached to an elliptic curve $E$ over $\Q$,
the assumption $\textup{(PO)}$ implies that $\# E(k_v)$ is prime to $p$ for all place $v$
in $K$ above $p$, where $k_v$ is the residue field of $K_v$, \ie $p$ is anomalous for $E$.\end{remark}

\begin{cor}\label{C:free}
If $S$ is an $n$-admissible set for $f$, then
\begin{mylist}
\item the natural map $\wh\Sel_\Delta^S(K_{\infty},\Tfn)\to \Sel_\Delta^S(K_m,\Tfn)$ is surjective.
\item $\wh\Sel_\Delta^S(K_{\infty},\Tfn)$ is free of
rank $\# S$ over $\Lambda \slash \varpi^n \Lambda$.
\end{mylist}
\end{cor}
\begin{proof}Note that \propref{P:step1} implies the corestriction map $\cores_m:\Sel_\Delta^S(K_{m+1},\Tfn)\to \Sel_\Delta^S(K_m,\Tfn)$ is surjective for all $m$ by a cardinality consideration, from which Part (1) follows.
Part (2) is an easy consequence of part (1) and \propref{P:step1}. \end{proof}

\begin{prop}\label{P:maximal}
If $S$ is an $n$-admissible set, then we have isomorphisms
\[\wh\Sel_\Delta^S(K_\infty,\Tfn)/\frakm_\Lam\iso \wh\Sel_\Delta^S(K,T_{f,1}),\quad\wh\Sel_\Delta^S(K_\infty,\Tfn)/\uf\Lam\iso \wh\Sel_\Delta^S(K_\infty,T_{f,1}).\]
\end{prop}
\begin{proof}
This is a consequence of the combination of \propref{P:step1} and \corref{C:free}.
\end{proof}

\subsection{Divisibility}
\subsubsection{}
Let $\varphi :\Lambda \to \cO_\varphi$ be an $\cO$-algebra homomorphism,
where $\cO_{\varphi}$ is a discrete valuation ring of characteristic $0$.
Let $\varpi_{\varphi}$ be an uniformizer of $\cO_\varphi$ and
$\mathfrak{m}_{\varphi}$ the maximal ideal of $\cO_{\varphi}$. If $M$ is a finitely generated $\cO_{\varphi}$-module, for each $x\in M$ we define
\[\Ord_{\uf_\vp}(x):=\sup\stt{m\in\Z_{\geq 0}\mid x\in \uf_{\varphi}^mM}.\]
It is clear that $x=0$ if and only if $\Ord_{\uf_\vp}(x)=\infty$. If $M$ is a $\Lam$-module, we let $M\ot_\varphi\cO_\varphi:=M\ot_{\Lam,\varphi}\cO_\varphi$.
\subsubsection{} 
For a finitely generated $\Lambda$-module $M$, we denote the characteristic ideal
attached to $M$ by $\mathrm{char}_{\Lambda}(M)$. Recall the following result of Bertolini-Darmon \cite[\proposition 3.1]{Bertolini_Darmon:IMC_anti}.
\begin{lm}\label{L:homomorphism}
Let $M$ be a finitely generated $\Lambda$-module and $L$ is an element of
$\Lambda$.
Suppose that for any homomorphism
$\varphi : \Lambda \to \cO_{\varphi}$, we have 
\[\length_{\cO_\varphi}(M\ot_\varphi\cO_\varphi)\leq \Ord_{\uf_\varphi}(\varphi(L)).\]
 Then $L\in \mathrm{char}_{\Lambda}(M)$.
\end{lm}
Let $n$ be a positive integer and let $\Delta>1$ be a square-free product of an odd number of prime factors which satisfies \defref{D:1} (1). For each $n$-admissible form $\cD =(\Delta,f_n)$ as in \defref{D:1},
we define two non-negative integers:
\begin{align*}
s_{\cD}= &\length_{\cO_{\varphi}}\Sel_\Delta(K_{\infty},\Afn)^\vee \otimes_\varphi\cO_\varphi;\\
t_{\cD}= &\Ord_{\varpi_\varphi}\varphi (\theta_\infty(\cD))\\
&\quad(\varphi(\theta_\infty(\cD))\in \cOn\powerseries{\Gamma}\ot_\varphi\cO_\varphi=\cO_\varphi/(\uf_\varphi^n)).
\end{align*}

The following key proposition is the analogue of \cite[\proposition 4.3]{Pollack_Weston:AMU}. The proof will be given in \secref{SS:proof}.
\begin{prop}\label{P:induction}Assume that \CR\,and \eqref{PO} hold. 
Let $t^*\leq n$ be a non-negative integer. Let $\cD_{t^*} =(\Delta,{f}_{n+t^*})$ be an $(n+t^*)$-admissible form and let $\cD_0=\cD_{t^*}\pmod{\uf^n}:=(\Delta,f_{n+t^*} \pmod{\varpi^n})$ be an $n$-admissible form.
Suppose that $t_{\cD_0}\leq t^*$. Then we have $s_{\cD_0}\leq 2t_{\cD_0}$.
\end{prop}
Note that if $\mathcal{D}=(\Delta,f_{n+t^*})$ is an $n+t^*$-adimissible form, then $\mathcal{D}\pmod{\varpi^n}=(\Delta,f_{n+t^*}\pmod{\varpi^n})$ is an $n$-admissible form.

Let $\pi$ be the unitary cuspidal automorphic representation of $\GL_2(\A)$ attached to the new form $f$.
 Let \[\theta_\infty:=\Theta_\infty(\pi,\bfone)\in\Lam\] be the theta element with trivial branch character defined in \cite[page 21]{Hsieh_Chida}.
\begin{prop}\label{P:R1}If $\Delta=N^-$, there exists an $n$-admissible form $\cD_n^f=(N^-,\bff_n^{\dagger,[k-2]})$ such that \[\theta(\cD_n^f)\con \theta_\infty\pmod{\uf^n}.\]
\end{prop}
\begin{proof}Let $\Delta=N^-$ and let $B$ be the definite quaternion algebra of absolute discriminant $N^-$. Let  $\lam_f^\circ:\bbT_B(N^+)\to\cO$ be the ring homomorphism defined by $\lam_f^\circ(T_\pme)=\al_\pme(f),\,\lam_f^\circ(S_\pme)=1$ if $\pme\ndivides N$ and $\lam_f^\circ(U_\pme)=\al_\pme(f)$ if $\pme\divides N$. By the Jacquet-Langlands correspondence, there exists a modular form $\bff\in \SBk(\wh R^\x_{N^+},\cO)$ such that $\bff\not \con 0\pmod{\uf}$ and $\bff$ is an eigenform of $\bbT_B(N^+)$ and $t \bff=\lam_f^\circ(t) \bff$ for all $t\in \bbT_B(N^+)$. Let $\bff^\dagger\in \SBk(\wh R^\x_{pN^+},\cO)$ be the $p$-stabilization of $\bff$ defined by
 \[\bff^\dagger(b)=\bff(b)-p^\frac{k-2}{2}A_p^{-1}\cdot\bff(b\pDII{p^{-1}}{1}).\]
The theta element $\theta_\infty$ is essentially constructed from the evaluation of $\bff^\dagger$ at Gross points (see \cite[Definition 4.1]{Hsieh_Chida}). Define $\bff_n^{\dagger,[k-2]}\in \cS^B(\cU_{N^+,p^n},\cOn)$ by
\[\bff_n^{\dagger,[k-2]}(b):=\sqrt{\beta}^\frac{2-k}{2}\cdot \pair{X^{k-2}}{\bff^\dagger(b)}_k\pmod{\uf^n}.\]
Following the argument in the proof of \cite[Theorem 5.7]{Hsieh_Chida}, one can show $\bff_n^{\dagger,[k-2]}\not\con 0\pmod{\uf}$, and $\cD_{n}^f:=(N^-,\bff_n^{\dagger,[k-2]})$ is the desired $n$-admissible form.
\end{proof}

We deduce our main theorem (\thmref{T:A}) from the above propositions.
\begin{thm}\label{T:main-theorem}
With the hypotheses \CR\,and \eqref{PO}, we have
$$
\mathrm{char}_{\Lambda}\Sel(K_\infty,A_f) \supset (L_p(K_{\infty},f)) .
$$
\end{thm}
\begin{proof}
Let  $\varphi :\Lambda \to \cO_{\varphi}$ be an $\cO$-algebra homomorphism. If $\varphi (L_p(K_\infty,f))=0$, then $\varphi (L_p(K_\infty,f))$
belongs to $\mathrm{Fitt}_{\cO_\varphi}(\Sel(K_\infty,A_f)^\vee\otimes_\Lambda \cO_\varphi)$
clearly.
Therefore we may assume that $\varphi (L_p(K_\infty,f))\neq 0$.
Choose $t^*$ larger than the $\cO_\varphi$-valuation of $\varphi (L_p(K_\infty,f))$.
For each positive integer $n$, consider the $(n+t^*)$-admissible form $\cD^f_{n+t^*}=(N^-,\bff_{n+t^*}^{\dagger,[k-2]})$ in \propref{P:R1}. Applying \propref{P:induction} to $\cD^f_{n+t^*}$ and $D^f_n=\cD^f_{n+t^*}\pmod{\uf^n}$, we find that
$\varphi (L_p(K_\infty,f))\pmod{\uf^n}=\varphi (\theta_\infty (\cD^f_n)^2)$
belongs to $\mathrm{Fitt}_{\cO_\varphi}(\Sel_{N^-}(K_\infty,\Afn)^\vee\otimes_\varphi\cO_\varphi)$ for all $\varphi$ and $n$. By \lmref{L:homomorphism}, $L_p(K_\infty,f)$ belongs to $\cap_{n=1}^\infty\Fitt_\Lam\Sel_{N^-}(K_\infty,\Afn)^\vee=\Fitt_\Lam\Sel_{N^-}(K_\infty,A_f)^\vee$. By \cite{Vatsal:nonvanishing} and \cite{Hsieh_Chida}, $L_p(K_\infty,f)\not =0$, so $\Sel_{N^-}(K_\infty,A_f)$ is $\Lam$-cotorsion. The theorem thus follows from \propref{P:3}.
\end{proof}

\begin{cor}\label{C:main}With the hypotheses \CR\,and \eqref{PO}, if the central $L$-value $L(f/K,\frac{k}{2})$ is non-zero, then the minimal Selmer group $\Sel(K,A_f)$ is finite. Assume further that $\ol{\rho}_f$ is ramified at all $\ell\divides N^-$. Then
\begin{align*}\length_\cO(\Sel(K,A_f))+\sum_{\ell\divides N^+}t_f(\ell)\leq &\Ord_\uf\left(\frac{L(f/\cK,\frac{k}{2})}{\Omega_f}\right),\end{align*}
where $t_f(\ell)$ are the Tamagawa exponent at $\ell$.
\end{cor}
\begin{proof}Note that $H^1_\Ord(K_\ell,A_f)=\stt{0}$ for $\ell\divides N^-$, so we have an exact sequence
\[0\to\Sel(K,A_f)\to \Sel_{N^-}(K,A_f)\stackrel{\gamma}\longto \prod_{\ell\divides N^+}H^1_\fin(K_\ell,A_f).\]
By the interpolation formula \eqref{E:interpolation} of $L_p(K_\infty,f)$ at the trivial character $\bfone$ and the fact $e_p(f,\bfone)$ is a \padic unit under \eqref{PO}, we find that
\[\bfone(L_p(K_\infty,f))=u\cdot \frac{L(f/\cK,\frac{k}{2})}{\Omega_f}\text{ for some }u\in\cO^\x.\]
 By the control theorem \propref{P:control1} and \thmref{T:main-theorem}, \[\length_\cO(\Sel_{N^-}(K,A_f))=\Ord_\uf\left(\frac{L(f/\cK,\frac{k}{2})}{\Omega_f}\right)<\infty. \]
In particular, $\Sel_{N^-}(K,A_f)$ is finite, and by \cite[\proposition 4.13]{Greenberg:Iwasawa_LMN_1716} the map $\gamma$ is surjective. Therefore, we have
\begin{align*}
\length_\cO(\Sel_{N^-}(K,A_f))
=&\length_\cO(\Sel(K,A_f))+\sum_{\ell\divides N^+}\length_\cO(H^1_\fin(K_\ell,A_f))\\
=&\length_\cO(\Sel(K,A_f))+\sum_{\ell\divides N^+}t_f(\ell).
 \end{align*}
This finishes the proof.
\end{proof}
\subsection{Proof of \propref{P:induction}}\label{SS:proof}
We will prove \propref{P:induction} by induction on $t_{\cD_0}$. If $t_{\cD_0}=\infty$ or $s_{\cD_0}=0$ then $s_{\cD_0}\leq 2t_{\cD_0}$ holds trivially. Therefore,
without loss of generality, we may assume that
\begin{itemize}
\item $t_{\cD_0}<\infty\iff \vp(\theta_\infty(\cD_0))\not =0$,
\item $s_{\cD_0}>0\iff \Sel_\Delta(K_{\infty},\Afn)\otimes_\varphi\cO_\varphi$ is non-trivial.
\end{itemize}
We write $t=t_{\cD_0}$. Consider the $(n+t)$-admissible form \[\cD:=(\Delta,f_{n+t^*} \pmod{\varpi^{n+t}}).\]
Let $\ell$ be an $(n+t)$-admissible prime which does not divide $\Delta$. Enlarge $\{\ell \}$ to an $(n+t)$-admissible
set $S$ with $(S,\Delta)=1$ and let \[\kappa_\cD (\ell )\in \wh\Sel_{\Delta\ell}(K_\infty,T_{f,n+t})\subset  \wh\Sel_{\Delta}^S(K_\infty,T_{f,n+t})\]
be the cohomology class attached to $\cD$ and $\ell$ constructed in \subsecref{SS:CohomologyClass}. By \corref{C:free}, the module $\cM_n:=\wh\Sel_\Delta^S(K_\infty,T_{f,n})\otimes_\varphi \cO_\varphi$ is free over $\cO_\varphi \slash \varphi(\varpi^n) \cO_\varphi$ for all $n$. Denote by $\kappa_{\cD, \varphi}(\ell)$ the image of $\kappa_\cD (\ell)$
in $\cM_{n+t}$ and let
\[e_\cD(\ell):=\mathrm{ord}_{\varpi_\varphi}(\kappa_{\cD, \varphi} (\ell))\footnote{The definition of $e_\cD(\ell)$ also depends on an auxiliary choice of $S$.}.\]
It follows from \thmref{T:first} that
$$
e_\cD(\ell)\leq
\mathrm{ord}_{\varpi_\varphi}(\partial_\ell \kappa_{\cD, \varphi} (\ell))
=\mathrm{ord}_{\varpi_\varphi}(\varphi(\theta_\infty(\cD)))=\ord_{\varpi_\varphi}(\varphi (\theta_\infty(\cD_0)))=t.
$$
Choose an element $\widetilde{\kappa}_{\cD, \varphi}(\ell)\in \cM_{n+t}$ which satisfies
$\varpi_\varphi^{e_\cD(\ell)}\cdot \widetilde{\kappa}_{\cD,\varphi}(\ell)={\kappa}_{\cD,\varphi}(\ell)$.
Note that $\widetilde{\kappa}_{\cD,\varphi}(\ell)$ is well-defined up to the kernel of the homomorphism $\cM_{n+t}\to \cM_n$. Let $\kappa'_{\cD, \varphi}(\ell)$ be the natural image of the cohomology class
$\widetilde{\kappa}_{\cD,\varphi}(\ell)$ in $\cM_n$.

\begin{lm}\label{L:properties}
The cohomology class $\kappa'_{\cD,\varphi}(\ell)\in\wh\Sel_\Delta^S(K_\infty,\Tfn)\ot_\varphi\cO_\varphi$ satisfies the following properties:
\begin{enumerate}
\item $\mathrm{ord}_{\varpi_\varphi}(\kappa'_{\cD,\varphi}(\ell))=0$,
\item $\mathrm{ord}_{\varpi_\varphi}(\partial_{\ell}(\kappa'_{\cD,\varphi}(\ell)))=
t-e_\cD(\ell)$,
\item $\partial_q(\kappa'_{\cD,\varphi}(\ell))=0$ for all $q\ndivides \Delta\ell$,
\item $\res_q(\kappa'_{\cD,\varphi}(\ell))\in\wh H^1_\Ord(K_{\infty,q},\Tfn)\ot_\varphi\cO_\varphi$ for all $q\divides \Delta\ell$.
\end{enumerate}
\end{lm}
\begin{proof}
 By \propref{P:maximal}, the map $\cM_{n+t}/\uf_\varphi\cM_{n+t}\to\cM_n/\uf_\varphi\cM_n$ is an isomorphism, so we have $\Ord_{\uf_\varphi}(\kappa'_{\cD, \varphi}(\ell))=\Ord_{\uf_\varphi}(\widetilde{\kappa}_{\cD, \varphi}(\ell))=0$. Part (1) follows immediately. Part (2) is a direct consequence of \thmref{T:first}.
Part (3) and (4) follow from the fact that $\kappa_{\cD} (\ell)$ belongs to
$\wh\Sel_{\Delta\ell}(K_\infty,T_{n+t})\ot_\varphi\cO_\varphi$ and the freeness result of the ordinary cohomology group at $\ell$ in \lmref{L:rank-one}.
\end{proof}
\begin{lm}\label{L:eta}
Let $\eta_\ell$ be the natural homomorphism
$$
\eta_\ell : \widehat{H}^1_{\sing}(K_{\infty,\ell}, \Tfn)\otimes_\varphi \cO_\varphi
\to \Sel_\Delta(K_\infty, \Afn)^\vee\otimes_\varphi \cO_\varphi
$$
sending $\kappa$ to $\eta_\ell(\kappa):s\mapsto\pair{\kappa}{v_\ell(s)}_\ell$. Then we have $\eta_\ell(\partial_\ell (\kappa'_{\cD,\varphi}(\ell)))=0$.
\end{lm}
\begin{proof}
Let $s\in\Sel_\Delta(K_\infty, \Afn)[\ker\varphi]$. By \lmref{L:properties} (3) (4), we see that $\langle \partial_q (\kappa'_{\cD,\varphi}(\ell)), v_q(s) \rangle_q=0$ for all $q\not =\ell$. The lemma thus follows from the global reciprocity law \eqref{E:globalRec}.
\end{proof}
\begin{lm}\label{L:trivial}
If $t=0$, then $s_{\cD_0}=0$, \ie $\Sel_\Delta(K_{\infty},\Afn)^\vee\otimes_{\varphi}\cO_\varphi$ is trivial.
\end{lm}
\begin{proof}
If $t=0$, then $\varphi (\theta_\infty (\cD_0))$ is a unit in $\cO_\varphi/(\uf_\varphi^n)$.
By \th \ref{T:first}, this implies that $\partial_\ell (\kappa_{\cD ,\varphi}(\ell))$ generates
$\widehat{H}^1_{\sing}(K_{\infty,\ell},T_n)\otimes_{\varphi}\cO_\varphi$ for any admissible prime $\ell$.
Therefore the map $\eta_\ell$ in the \lmref{L:eta} is trivial for all admissible primes.
Assume that $\Sel_\Delta(K_{\infty},\Afn)^\vee\otimes_{\varphi}\cO_\varphi$ is non-trivial.
By Nakayama's lemma,
$$
(\Sel_\Delta(K_{\infty},\Afn)^\vee
\otimes_{\varphi}\cO_\varphi) \slash \mathfrak{m}_{\cO_\varphi}
=(\Sel_\Delta(K_{\infty},\Afn)[ \mathfrak{m}_\Lambda])^\vee
\otimes_{\varphi}\cO_\varphi \slash \mathfrak{m}_{\cO_\varphi}$$
is non-zero.
Let $s$ be a non-trivial element in
$\Sel_\Delta(K_{\infty},\Afn)[ \mathfrak{m}_\Lambda]$.
By \propref{P:control1}, $\Sel_\Delta(K_{\infty},\Afn)[ \mathfrak{m}_\Lambda]\subset\Sel_\Delta(K,A_{f,1})$, so $s$ can be viewed as an element in
$H^1(K,A_{f,1})$. By \thmref{T:Theorem3.2}, we can choose an $n$-admissible prime $\ell\ndivides \Delta$
such that $\partial_\ell(s)=0$ and $v_{\ell}(s)\neq 0$. Since the local Tate pairing $\pairing_\ell$ is perfect, $\eta_{\ell}$ is non-zero.
This is a contradiction.
\end{proof}
In view of \lmref{L:trivial}, we may assume that $t>0$. Let $\Pi$ be the set of rational primes $\ell$ satisfying the following conditions:
\begin{mylist}
\item $\ell$ is $(n+t^*)$-admissible and $\ell\ndivides \Delta$,
\item The integer $e_\cD(\ell)=\Ord_{\uf_\varphi}(\kappa_{\cD,\varphi}(\ell))$ is minimal among $\ell$ satisfying the condition (1).\end{mylist}
Then $\Pi \neq \emptyset$ by \thmref{T:Theorem3.2}. Let $e=e_\cD(\ell)$ for any $\ell\in\Pi$.
\begin{lm}\label{L:3}
We have $e<t$.
\end{lm}
\begin{proof}
Suppose that $e=t$. Then $e_\cD(\ell)=t$ for all $(n+t^*)$-admissible
primes $\ell$ since $e_\cD(\ell)\leq t$.
By \propref{P:control1}, $H^1(K,A_1)\to H^1(K_\infty,A_n)[\mathfrak{m}_\Lambda]$ is an isomorphism.
Hence there exists a non-zero element \[s\not =0\in \Sel_\Delta(K_{\infty},\Afn)[\frakm_{\Lam}]\subset H^1(K,A_1)
\otimes_{\varphi}\cO_\varphi.\]
By \thmref{T:Theorem3.2} (1), there exists an $(n+t^*)$-admissible prime $\ell$ such that $v_\ell(s)$
is non-zero in $H_\fin^1(K,A_1)\otimes_\varphi \cO_\varphi$.
On the other hand, by \lmref{L:properties} (4), the image of $\partial_\ell (\kappa'_{\cD,\varphi} (\ell))$
in $H^1(K_\ell,T_{f,1})\otimes_\varphi \cO_\varphi$ is non-zero. Moreover, by \lmref{L:eta}, the image of $\partial_\ell (\kappa'_{\cD,\varphi} (\ell))$
in $H^1(K_\ell,T_{f,1})\otimes_\varphi \cO_\varphi$ is orthogonal
to $v_{\ell}(s)$ with respect to the local Tate pairing. Since the local Tate pairing
$$
H_\fin^1(K,A_{f,1})\otimes_\varphi \cO_\varphi \times H_{\sing}^1(K,T_{f,1})\otimes_\varphi \cO_\varphi \to
\cO_{\varphi} \slash \varpi_\varphi \cO_\varphi
$$
is perfect and $H_\fin^1(K,A_{f,1})\otimes_\varphi \cO_\varphi$ and $H_{\sing}^1(K,T_{f,1})\otimes_\varphi \cO_\varphi$
are one dimensional vector space over $\cO_{\varphi} \slash \varpi_\varphi \cO_\varphi$,
this is a contradiction.
\end{proof}

Let $\ell_1 \in \Pi$ and $S$ be an $(n+t^*)$-admissible set containing $\ell_1$.
Let $\kappa_1$ be the image of $\kappa_{\cD,\varphi}'(\ell_1)$ in
\[\wh\Sel_\Delta^S(K_\infty, \Tfn)\otimes_{\varphi} \cO_{\varphi} \slash \varpi_\varphi \cO_\varphi
= \wh\Sel_\Delta^S(K_\infty, \Tfn)\slash \mathfrak{m}_{\Lambda}
\otimes_{\varphi} \cO_{\varphi}\hookto H^1(K,T_{f,1})\otimes_{\varphi}\cO_{\varphi},\]
where the last map induced by the corestriction is injective by \propref{P:maximal}. Hence, $\kappa_1$ is non-zero element in $H^1(K,T_{f,1})\otimes_{\varphi}\cO_{\varphi}$.
By the first part of \thmref{T:Theorem3.2}, there exists an $(n+t^*)$-admissible prime $\ell_2\ndivides\Delta$
such that $\partial_{\ell_2}(\kappa_1)=0$ and \[v_{\ell_2}(\kappa_1)\neq 0 \in H^1_\fin(K,T_{f,1})\otimes_{\varphi} \cO_{\varphi}.\]
It follows from the fact $v_{\ell_2}(\kappa_1)\not =0$ and the minimality of $e=e_\cD(\ell_1)=\ord_{\varpi_\varphi}(\kappa_{\cD ,\varphi}(\ell_1))$ that
$$
\ord_{\varpi_\varphi}(v_{\ell_2}(\kappa_{\cD ,\varphi}(\ell_1)))=\ord_{\varpi_\varphi}(\kappa_{\cD ,\varphi}(\ell_1))\leq \ord_{\varpi_\varphi}(\kappa_{\cD ,\varphi}(\ell_2))
\leq \ord_{\varpi_\varphi}(v_{\ell_1}(\kappa_{\cD ,\varphi}(\ell_2))).
$$
(the last inequality is due to the fact that $v_{\ell_1}$ is a homomorphism). By the second explicit reciprocity law (\thmref{T:second}), there exists an $(n+t^*)$-admissible form $\cD''_{t^*}=(\Delta\ell_1\ell_2,g_{n+t^*})$ such that \[v_{\ell_2}(\kappa_\cD(\ell_1))=v_{\ell_1}(\kappa_{\cD ,\varphi}(\ell_2))=\theta_{\infty}(\cD''_t)\quad(\cD_t''=(\Delta\ell_1\ell_2,g_{n+t^*}\pmod{\uf^{n+t}})).\]
In particular, $\ord_{\varpi_\varphi}(v_{\ell_1}(\kappa_{\cD ,\varphi}(\ell_2)))=\ord_{\varpi_\varphi}(v_{\ell_1}(\kappa_{\cD ,\varphi}(\ell_1))).$
We thus conclude that \[\ord_{\varpi_\varphi}(v_{\ell_2}(\kappa_{\cD ,\varphi}(\ell_1)))=e_\cD(\ell_1)=e_\cD(\ell_2)=e\text{ and }\ell_2 \in \Pi.\]   Let $\cD_0'':=(\Delta\ell_1\ell_2,g_{n+t^*} \pmod{\varpi^n})$.
Then we have \[t_{\cD_0''}=\Ord_{\uf_\varphi}(\varphi(\theta_\infty(\cD_0'')))=e<t=t_{\cD_0}\leq t^*.\]
Therefore, we can apply the induction hypothesis to $\cD''_{t^*}$ and conclude that $s_{\cD_0''} \leq 2t_{\cD_0''}$. To finish the proof, it suffices to show
\beq\label{E:induction}s_{\cD_0}\leq s_{\cD_0''}+2(t-t_{\cD_0''}).\eeq

Let $S_{[\ell_1 \ell_2]}$ denote the subgroup of $\Sel_\Delta(K_\infty,\Tfn)$ consisting of classes which are locally trivial at the primes dividing $\ell_1$ and $\ell_2$. By definition, there are two exact sequences of $\Lam$-modules:
\begin{align}\label{E:exact1}\wh H^1_\sing(K_{\infty,\ell_1},\Tfn)\oplus &\wh H^1_\sing(K_{\infty,\ell_2},\Tfn)\stackrel{\eta_s}\longto \Sel_\Delta(K_\infty,\Afn)^\vee\to S_{[\ell_1\ell_2]}^\vee\to 0\\
\intertext{ and }
\label{E:exact2}\wh H^1_\fin(K_{\infty,\ell_1},\Tfn)\oplus &\wh H^1_\fin(K_{\infty,\ell_2},\Tfn)\stackrel{\eta_f}\longto \Sel_{\Delta\ell_1\ell_2}(K_\infty,\Afn)^\vee\to S_{[\ell_1\ell_2]}^\vee\to 0,\end{align}
where $\eta_s$ and $\eta_f$ are induced by the local pairing $\pairing_{\ell_1}\oplus\pairing_{\ell_2}$. Let $\eta_s^\varphi$ (resp. $\eta_f^\varphi$) denote the map induced from $\eta_s$ (resp. $\eta_f^\varphi$) after tensoring with $\cO_\varphi$ via $\varphi$.
Fixing an isomorphism $\oplus_{i=1}^2\wh H^1_\sing(K_{\infty,\ell_i}\Tfn)\ot_\varphi\cO_\varphi\iso \cO_\varphi^{\oplus 2}$, from \lmref{L:eta} we deduce that $\eta_s^\varphi$ factors through the quotient
\[\cO_\varphi/(\partial_{\ell_1} (\kappa'_{\cD, \varphi}(\ell_1)))\oplus \cO_\varphi/(\partial_{\ell_2} (\kappa'_{\cD, \varphi}(\ell_2))).\]
Moreover, by \lmref{L:properties} (4), we have
$$
t-t_{\cD_0''}=\mathrm{ord}_{\varpi_\varphi}(\partial_{\ell_1} \kappa'_{\cD, \varphi}(\ell_1))
=\mathrm{ord}_{\varpi_\varphi}(\partial_{\ell_2} \kappa'_{\cD, \varphi}(\ell_2)).
$$
Hence, from \eqref{E:exact1} we obtain the exact sequence
\beq\label{E:exact3}\left(\cO_\varphi/(\uf_\varphi^{t-t_{\cD_0''}})\right)^{\oplus 2}\stackrel{\eta_s^\varphi}\longto \Sel_\Delta(K_\infty,\Afn)^\vee\ot_\varphi\cO_\varphi\to S_{[\ell_1\ell_2]}^\vee\ot_\varphi\cO_\varphi\to 0.\eeq
\begin{lm}\label{L:4}
The kernel of $\eta_f^\varphi$ contains the elements
$(0,v_{\ell_2}(\kappa'_{\cD ,\varphi}(\ell_1)))$ and $(v_{\ell_1}(\kappa'_{\cD ,\varphi}(\ell_2)),0)$.
\end{lm}
\begin{proof}Let $s\in\Sel_{\Delta\ell_1\ell_2}(K_\infty,\Afn)[\ker\varphi]$. By \lmref{P:annihilator_ord} and \lmref{L:properties}
 (3) (4), we have
 \[\pair{\partial_q(\kappa'_{\cD ,\varphi}(\ell_1))}{v_q(s)}_q=0\text{ for }q\ndivides\Delta\ell_1\text{ and }\pair{\res_{\ell_1}(\kappa'_{\cD ,\varphi}(\ell_1))}{\res_{\ell_1}(s)}_{\ell_1}=0.\]
 By the global reciprocity law, we find that $\pair{v_{\ell_2}(\kappa'_{\cD ,\varphi}(\ell_1))}{\res_{\ell_2}(s)}_{\ell_2}=0$. The same argument shows that $\pair{v_{\ell_1}(\kappa'_{\cD ,\varphi}(\ell_2))}{\res_{\ell_1}(s)}_{\ell_1}=0$. This completes the proof.
\end{proof}
Fixing an isomorphism $\oplus_{i=1}^2\wh H^1_\fin(K_{\infty,\ell_i}\Tfn)\ot_\varphi\cO_\varphi\iso \cO_\varphi^{\oplus 2}$, from \eqref{E:exact2} and \lmref{L:4} we deduce the exact sequence
\[\cO_\varphi/(v_{\ell_2}(\kappa'_{\cD ,\varphi}(\ell_1)))\oplus\cO_\varphi/(v_{\ell_1}(\kappa'_{\cD ,\varphi}(\ell_2))) \stackrel{\eta_f^\varphi}\longto \Sel_{\Delta\ell_1\ell_2}(K_\infty,\Afn)^\vee\ot_\varphi\cO_\varphi\to S_{[\ell_1\ell_2]}^\vee\ot_\varphi\cO_\varphi\to 0.\]
Note that
$$
\mathrm{ord}_{\varpi_\varphi}(v_{\ell_2}(\kappa'_{\cD,\varphi}(\ell_1)))
=\mathrm{ord}_{\varpi_\varphi}(v_{\ell_1}(\kappa'_{\cD,\varphi}(\ell_2)))=t_{\cD_0''}-e=0.
$$
We thus find that
\beq\label{E:exact4}\Sel_{\Delta\ell_1\ell_2}(K_\infty,\Afn)^\vee\ot_\varphi\cO_\varphi\isoto S_{[\ell_1\ell_2]}^\vee\ot_\varphi\cO_\varphi.\eeq
Now it is clear that \eqref{E:induction} follows from \eqref{E:exact3} and \eqref{E:exact4}.

\begin{thank}
This work began during the first authur's visit to Taida Institute of Mathematical Science.
He is grateful to their hospitality and financial support.
The authors thank the referee for helpful suggestions. 
\end{thank}

\bibliographystyle{amsalpha}
\bibliography{mybib}
\end{document}